\documentclass{article}

\usepackage{authblk}
\usepackage{amsmath,amsthm,amsfonts,graphicx,amssymb}
\usepackage{hyperref}

\title{A coarse geometric approach \\ to graph layout problems}
\author[1]{Wanying Huang}
\author[2]{David Hume\thanks{Email address for correspondence: d.hume@bham.ac.uk}}
\author[3]{Samuel~J.~Kelly}
\author[3]{Ryan Lam}
\affil[1]{McGill University, Montreal, Canada}
\affil[2]{University of Birmingham, UK }
\affil[3]{University of Bristol, UK}
\date{\today}                     
\numberwithin{equation}{section}
\newtheorem{theorem}[equation]{Theorem}
\newtheorem{proposition}[equation]{Proposition}
\newtheorem{corollary}[equation]{Corollary}
\newtheorem{lemma}[equation]{Lemma}

\theoremstyle{definition}

\newtheorem{question}[equation]{Question}

\newtheorem{definition}[equation]{Definition}
\newtheorem{remark}[equation]{Remark}

\newtheorem*{corollary*}{Corollary}
\newtheorem*{theorem*}{Theorem}
\newtheorem*{proposition*}{Proposition}

\newtheoremstyle{citing}
  {3pt}
  {3pt}
  {\itshape}
  {}
  {\bfseries}
  {}
  {.5em}
  {\thmnote{#3}}

\theoremstyle{citing}

\DeclareMathOperator{\bw}{bw}
\DeclareMathOperator{\cw}{cw}

\DeclareMathOperator{\mla}{mla}
\DeclareMathOperator{\pw}{pw}
\DeclareMathOperator{\sn}{sn}
\DeclareMathOperator{\tbw}{tbw}
\DeclareMathOperator{\tw}{tw}

\DeclareMathOperator{\vs}{vs}
\DeclareMathOperator{\sumcut}{sc}

\DeclareMathOperator{\wid}{wid}
\DeclareMathOperator{\diam}{diam}

\DeclareMathOperator{\cut}{cut}
\DeclareMathOperator{\sep}{sep}
\DeclareMathOperator{\im}{im}
\DeclareMathOperator{\vol}{vol}
\DeclareMathOperator{\wir}{wir}
\DeclareMathOperator{\para}{par}

\newcommand{\set}[1]{\left\{#1\right\}}
\newcommand{\setcon}[2]{\left\{#1\ \left|\ #2\right.\right\}}
\newcommand{\abs}[1]{\left\lvert#1\right\rvert}

\newcommand{\R}{\mathbb{R}}

\newcommand{\N}{\mathbb{N}}
\newcommand{\Z}{\mathbb{Z}}
\newcommand{\HH}{\mathbb{H}}

\def\XXint#1#2#3{{\setbox0=\hbox{$#1{#2#3}{\int}$}
\vcenter{\hbox{$#2#3$}}\kern-.5\wd0}}

\numberwithin{equation}{section}

\begin{document}

\maketitle

\begin{abstract}
We define a range of new coarse geometric invariants based on various graph-theoretic measures of complexity for finite graphs, including: treewidth, pathwidth, cutwidth and bandwidth. We prove that, for bounded degree graphs, these invariants can be used to define functions which satisfy a strong monotonicity property, namely they are monotonically non-decreasing with respect to a large-scale geometric generalisation of graph inclusion, and as such have potential applications in coarse geometry and geometric group theory. On the graph-theoretic side, we prove asymptotically optimal bounds on most of the above widths for the family of all finite subgraphs of any bounded degree graph whose separation profile is known to be of the form $r^a\log(r)^b$ for some $a>0$. This large class includes Diestel-Leader graphs, all Cayley graphs of non-virtually cyclic polycyclic groups, uniform lattices in almost all connected unimodular Lie groups, and many hyperbolic groups.
\end{abstract}

\section{Introduction}
A core aspect of coarse graph theory is to identify geometric properties of (typically connected, infinite) graphs which are preserved by a large-scale geometric generalisation of graph isomorphism known as quasi-isometry, and use these properties to distinguish classes of graphs. A particular focus has been on finding ``large scale'' analogues of famous graph-theoretic results, such as Menger's theorem, K\"onig's lemma, Halin's grid theorem, etc.\ and studying the relationship between variants of Gromov’s asymptotic dimension and graph-theoretic notions such as graph minors (cf.\ \cite{GeorPapo} and references therein).

In this paper, we are interested in a large-scale geometric generalisation of graph inclusion, called a \textbf{regular embedding}. To consider inclusion of graphs from a large-scale perspective, it is necessary to introduce some additional restrictions, as every graph can be included into a complete graph, and all large-scale geometric information is lost. To deal with this, we will restrict our attention to bounded degree graphs. In this context, the definition of a regular embedding is very simple: given two bounded degree graphs $X$ and $Y$ and a constant $\kappa\geq 1$, a regular map $\psi:VX\to VY$ is \textbf{$\kappa$-regular} if
\begin{itemize}
    \item $d_Y(\psi(v),\psi(v'))\leq \kappa d_X(v,v')$, for all $v,v'\in VX$; and
    \item $|\psi^{-1}(w)|\leq \kappa$, for every $w\in VY$.
\end{itemize}
We say a map is a regular embedding if it is $\kappa$-regular for some $\kappa$. Regular embeddings are a common geometric generalisation of graph inclusion among graphs of bounded degree, and subgroup inclusion among finitely generated groups (with respect to word metrics). We will use $d_X$ to denote the (extended) metric on $X$ where $d_X(x,y)$ is the length of the shortest path from $x$ to $y$ in $X$ (or $+\infty$ if no such path exists). Quasi-isometries, and to a significantly lesser extent regular embeddings, have been intensively studied in the field of (Cayley graphs of) finitely generated groups, as they are natural geometric generalisations of commensurability and subgroup inclusion, respectively. Our interest is in finding geometric properties of graphs and metric spaces that behave sufficiently well that we may use them to prove the non-existence of regular embeddings. The two classical examples are the growth function and asymptotic dimension which obstruct regular embeddings from the $3$-regular tree to a $2$-dimensional integer grid and vice-versa respectively. Over the last ten years, the number and variety of these invariants has increased dramatically, linking regular embeddings to expander graphs, fractal geometry, Lie theory, isoperimetry, Lorentz geometry, electrical capacitance, and more \cite{BarrettHume,Frances-coarse, GladShum, LeCozGournay, HumeMackTess-Pprof, HumeMackTess-PprofLie, HumeMackTess-genasdim, HumSepExp}.
\medskip

For this paper, the most important of these new invariants is the \textbf{separation profile} of Benjamini-Schramm-Tim\'ar \cite{BenSchTim-12-separation-graphs}. We recall that, for a finite graph $\Gamma=(V\Gamma,E\Gamma)$ and a constant $\varepsilon\in(0,1)$, the \textbf{$\varepsilon$-cutsize} of $\Gamma$ is the minimal cardinality of a subset $S\subseteq V\Gamma$ such that every connected component of $\Gamma-S$ has at most $\varepsilon |V\Gamma|$ vertices.

The $\varepsilon$-separation profile of a (typically infinite) graph $X$ is the function $\sep^\varepsilon_X:\N\to\N$ given by
\[    \sep^\varepsilon_X(r)=\max\setcon{\cut^\varepsilon(\Gamma)}{\Gamma\leq X,\ |V\Gamma|\leq r}
\]
Here we use the notation $\Gamma\leq X$ to signify that $\Gamma$ is a subgraph of $X$. We consider separation profiles (and indeed many other functions in this paper) using a natural partial order: given two functions $f,g:A\to B$ where $A,B\subseteq [0,+\infty)$, we write $f\lesssim g$ if there is a constant $C\geq 1$ such that for every $a\in A$
\[
 f(a) \leq Cg(Ca)+C
\]
We write $f\simeq g$ if $f\lesssim g$ and $g\lesssim f$. This partial order is common in coarse geometry, one classical example is the growth function of a finitely generated group: the exact value of the function depends on the choice of finite generating set, but they are all $\simeq$-equivalent, allowing one to define the growth function of the group as the $\simeq$ equivalence class of the growth function of any one of its Cayley graphs. Where appropriate, we will write $\lesssim_{c,d,\ldots}$ and $\simeq_{c,d,\ldots}$ to indicate that $C$ depends only on the parameters $c,d,\ldots$.

The separation profile enjoys two important properties:
\begin{proposition}\textup{\cite[pp.642]{BenSchTim-12-separation-graphs}}
    For any graph $X$ and any $\varepsilon,\varepsilon'\in(0,1)$,
    \[
        \sep^\varepsilon_X\simeq_{\varepsilon,\varepsilon'} \sep^{\varepsilon'}_X 
    \]
\end{proposition}
When no constant $\varepsilon$ is specified it is assumed to be $\frac12$.

\begin{proposition}\label{monotone}\textup{\cite[Lemma 1.3]{BenSchTim-12-separation-graphs}}
    If $X,Y$ are graphs with finite maximum degrees $\Delta_X,\Delta_Y$ respectively and there is a regular map $\phi:VX\to VY$, then
\[
    \sep_X \lesssim_{\Delta_X,\Delta_Y} \sep_Y
\]
\end{proposition}

Our main result is that Proposition \ref{monotone} holds when we replace the cutsize by many other natural measures of complexity associated to finite graphs. 

\subsection{Coarse monotonicity of graph layout problems}
The primary objects of study in this paper are \textbf{graph layout problems}: constructing linear orderings or decompositions of graphs in order to minimise a specified cost function. Examples of such invariants include treewidth, cutwidth and bandwidth.  There are many motivations for finding upper bounds on these invariants coming from a variety of applications including: optimization of networks for parallel computer architectures, VLSI circuit design, information retrieval, numerical analysis, computational biology, graph theory, scheduling and archaeology (cf.\ \cite{DPS-survey}, \cite{Bodlaender_treewidth} and references therein). Typically, realising the minimal possible value of these invariants is NP-hard (see, for example \cite{GJ79} for cutwidth and \cite{ACP-twNP} for treewidth). Despite this, there are asymptotically optimal upper bounds for large classes of graphs. One noteworthy result in the context of this paper is that for every $m$-vertex graph $H$, any $n$-vertex graph which does not contain $H$ as a minor has treewidth at most $m^{3/2}n^{1/2}$ \cite[(1.8)]{Alon-Seymour-Thomas}. 

We will split the graph layout problems we are interested in into two families and (initially) consider them separately.

\subsubsection{Ordering problems: cutwidth, bandwidth...}
A linear ordering of a finite graph $\Gamma$ is simply an enumeration of its vertices, i.e.\ a bijection $f:V\Gamma\to \{1,\ldots,|V\Gamma|\}$. There are various criteria which could qualify an enumeration as ``good''; we could seek an enumeration which minimises one of the following two functions
\[
\begin{array}{lll}
    \cw_f:\set{1,\ldots,|V\Gamma|}\to \N & \textup{given by} & \cw_f(i)=\abs{\setcon{xy\in E\Gamma}{f(x)<i\leq f(y)}} \\
   \bw_f:E\Gamma\to\N & \textup{given by} & \bw_f(xy)=|f(x)-f(y)|
\end{array}
\]
Specifically, the cutwidth ($\cw$), bandwidth ($\bw$), sumcut ($\sumcut$) and minimal linear arrangement ($\mla$) of a finite graph $\Gamma$ are defined respectively as 
\begin{align*}
    \cw(\Gamma) & = \min_f\left(\max_i \cw_f(i)\right) =\min_f \|\cw_f\|_{\infty} \\
    \bw(\Gamma) & = \min_f\left(\max_{xy\in E\Gamma} \bw_f(xy)\right) =\min_f \|\bw_f\|_{\infty} \\
    \sumcut(\Gamma) & = \min_f\left(\sum_i \cw_f(i)\right) =\min_f \|\cw_f\|_{1} = \\
    \mla(\Gamma) & = \min_f\left(\sum_{xy\in E\Gamma} \bw_f(xy)\right) =\min_f \|\bw_f\|_{1} 
\end{align*}
Extending this, for each $p\in [1,+\infty]$ we define the $p$-cutwidth and $p$-bandwidth respectively of a finite graph to be
\begin{align*}
    \cw^p(\Gamma) & =\min_f \|\cw_f\|_{p} \\
    \bw^p(\Gamma) & =\min_f \|\bw_f\|_{p}
\end{align*}
Then, for a (typically infinite) graph $X$, we define the $p$-cutwidth and $p$-bandwidth profiles of $X$ to be
\begin{align*}
    \cw^p_X(r) & = \max\setcon{\cw^p(\Gamma)}{\Gamma\leq X,\ |V\Gamma|\leq r} \\
    \bw^p_X(r) & = \max\setcon{\bw^p(\Gamma)}{\Gamma\leq X,\ |V\Gamma|\leq r}
\end{align*}

\begin{remark}
    We will typically use $\cw_X$ and $\bw_X$ in place of $\cw^\infty_X$ and $\bw^\infty_X$.
\end{remark}

All of these profiles enjoy the same monotonicity under regular maps as the separation profile.

\begin{theorem}\label{thm:monotoneLA} Let $X,Y$ be bounded degree graphs with maximum degrees $\Delta_X,\Delta_Y$ respectively. If there is a $\kappa$-regular map $\phi:VX\to VY$, then 
\[
 \cw_X\lesssim_{\kappa,\Delta_X,\Delta_Y} \cw_Y \quad \textup{and} \quad \bw_X\lesssim_{\kappa,\Delta_X,\Delta_Y} \bw_Y
\]
If, in addition, $EY$ is infinite, then for every $p\in [1,+\infty)$ we have
\[
 \cw_X^p\lesssim_{\kappa,\Delta_X,\Delta_Y,p} \cw_Y^p \quad \textup{and} \quad \bw_X^p\lesssim_{\kappa,\Delta_X,\Delta_Y,p} \bw_Y^p
\]    
\end{theorem}
As in the case of the separation profile, the finite maximum degree assumption is essential: any injective map from the countably infinite complete graph to a countable star is regular. The one exception is for usual bandwidth, but this is not interesting in our context as $\bw_X(r)\simeq r$ whenever $X$ does not have bounded degree. To see this, note that the upper bound $\bw(\Gamma)\leq |V\Gamma|-1$ holds for any finite graph, while the star graph on $r$ vertices has bandwidth $\lceil\frac12(r-1)\rceil$.

\subsubsection{Decomposition problems: treewidth, pathwidth,...}
The next class of invariants we consider are based on graph decompositions:

\begin{definition}
    Let $G$ be a (possibly infinite) graph and let $\Gamma$ be a finite graph. A $G$-decomposition of $\Gamma$ is a family $\mathcal X=\{X_g\}_{g\in VG}$ of subsets of $V\Gamma$ satisfying the three conditions
    \begin{enumerate}
        \item $\bigcup_{g\in G} X_g=V\Gamma$,
        \item for every edge $xy\in E\Gamma$, there is some $g\in VG$ such that $x,y\in X_g$,
        \item for each $v\in V\Gamma$, the full subgraph of $G$ with vertex set $\setcon{g\in VG}{v\in X_g}$ is connected.
    \end{enumerate}
    We define $\wid_{\mathcal X}:VG\to\N$ by $\wid_{\mathcal X}(g)=|X_g|$.
\end{definition}
We define the \textbf{$(G,p)$-width of $\Gamma$} to be
\[
\wid^{G,p}(\Gamma) = \min_{\mathcal X} \|\wid_{\mathcal X}\|_p  = \min_{\mathcal X} \left(\sum_{g\in VG} |X_g|^p\right)^\frac1p
\]
where the minimum is taken over all $G$-decompositions of $\Gamma$. Note that as $\Gamma$ is finite and $G$ is non-empty, there is always a trivial $G$-decomposition $X_g=V\Gamma$ for some $g\in VG$ and $X_h=\emptyset$ for all $h\in VG\setminus\{g\}$, so for every $G,\Gamma,p$, we have $\wid_G^p(\Gamma)\leq |V\Gamma|$.

When we replace $G$ by the regular tree of countably infinite valence $T$ or a biinfinite countable path $P$, we recover the treewidth $\tw(\Gamma)=\wid^{T,\infty}(\Gamma)-1$ and pathwidth $\pw(\Gamma)=\wid^{P,\infty}(\Gamma)-1$ respectively.

The \textbf{$(G,p)$-width profile} of a (typically infinite) graph $X$ is given by
\[
 \wid^{G,p}_X(r)=\max\setcon{\wid^{G,p}(\Gamma)}{\Gamma\leq X,\ |V\Gamma|\leq r}
\]
Again, all of these profiles enjoy the same monotonicity under regular maps enjoyed by the separation profile.

\begin{theorem}\label{thm:monotonedecomp} Let $X,Y$ be bounded degree graphs with maximum degrees $\Delta_X,\Delta_Y$ respectively. If there is a regular map $\phi:VX\to VY$, then for every graph $G$, and every $p\in[1,\infty]$ we have
\[
 \wid^{G,p}_X\lesssim_{\Delta_X,\Delta_Y,G,p} \wid^{G,p}_Y
\]
\end{theorem}

Graph width profiles also behave well with respect to regular maps between the decomposing graphs when they have bounded degree.

\begin{proposition}\label{prop:widthGinv}
    Let $G,G'$ be bounded degree graphs. If there is a $\kappa$-regular map $\phi:VG\to VG'$, then for every $p\in[1,+\infty]$ and every bounded degree graph $X$
    \[
        \wid^{G',p}_X(r) \lesssim_{p,\kappa,\Delta_G,\Delta_{G'}} \wid^{G,p}_X(r)
    \]
\end{proposition}

We define the treewidth and pathwidth profiles of a (typically infinite) graph $X$ to be
\begin{align*}
    \tw_X(r) &= \max\setcon{\tw(\Gamma)}{\Gamma\leq X,\ |V\Gamma|\leq r} \\
    \pw_X(r) &= \max\setcon{\pw(\Gamma)}{\Gamma\leq X,\ |V\Gamma|\leq r}
\end{align*}
It is clear that $\tw_X(r)\simeq \wid^{T,\infty}_X(r)$ and $\pw_X(r)\simeq \wid^{P,\infty}_X(r)$ for every graph $X$, and therefore $\tw_X$ and $\pw_X$ also behave monotonically with respect to regular maps.

\subsection{Comparing invariants}
Using existing results from the graph theory literature we deduce the following:

\begin{proposition}\label{prop:compare}
    Let $X$ be a bounded degree graph. Then
    \[
    \sep_X \lesssim \tw_X \lesssim \pw_X\simeq \cw_X \lesssim \bw_X
    \]
\end{proposition}
When $X$ is the infinite $3$-regular tree, $\sep_X\simeq \tw_X\simeq 1$, $\pw_X(r)\simeq \cw_X(r)\simeq \ln(1+r)$ and $\bw_X(r)\simeq r/\ln(1+r)$. We are not aware of a bounded degree graph $X$ satisfying $\sep_X\not\simeq \tw_X$ (see question \ref{qu:septw}).

\begin{remark}
    As a byproduct of the proof of Proposition \ref{prop:compare}, we actually show (Remark \ref{rem:otherinvs}) that analogously defined profiles using topological bandwidth, search number and vertex separation in place of cutwidth are $\simeq$-equivalent to $\cw_X$. These graph invariants will not be used in the paper, so we will not recall the definitions.
\end{remark}
Of the profiles introduced so far in this paper, the best understood (at least in the context of Cayley graphs) is the separation profile. Taking this as our motivation, we prove upper and lower bounds on $p$-cutwidth profiles in terms of separation that differ by at most a logarithmic function.

\begin{theorem}\label{sepboundscw}
 Let $X$ be a (typically infinite) graph with maximum degree $\Delta_X$. Then, for every $p\in[1,+\infty]$
    \[
     r^{\frac1p}\sep_X(r) \lesssim \cw^p_X(r) \lesssim r^{\frac1p}\cw_X(r) \lesssim r^{\frac1p}\sep_X(r) (1+\log_2(r))
    \]
\end{theorem}
Here we adopt the convention that $\frac{1}{+\infty} = 0$. While this means that, for the moment, cutwidth profiles obstruct the same regular maps as separation profiles, it is excellent for computing $p$-cutwidth profiles.

\subsection{Calculating $p$-cutwidth}

We start by briefly recapping some highlights from the theory of separation profiles.

\begin{itemize}
    \item Let $X$ be a vertex transitive graph\footnote{if $X$ is vertex transitive and $\gamma_X(r)\lesssim r^{d'}$ for some $d'>0$, then $\gamma_X(r)\simeq r^{d}$ for some $d\in\N$ \cite{Trofimov,Guivarch}} with $\gamma_X(r)\simeq r^d$, then
    \[
     \sep_X(r)\simeq r^{1-\frac1d}
    \]
    (see \cite{HumeMackTess-Pprof,LeCozGournay})
    \item Let $X$ be a Cayley graph of a uniform lattice in a connected semisimple Lie group $G\neq SL(2,\R)$ of real rank $r$. Then
    \[
    \sep_X(r) \simeq 
    \left\{
    \begin{array}{ll}
        r^{1-\frac{1}{Q_G}} &\quad \textup{if }r=1 \\
        r/\log(r) & \quad \textup{if }r\geq 2
    \end{array}
    \right.
    \]
    where $Q_G\in\N$ is the conformal dimension of the boundary of the symmetric space $\HH^m_{\mathbb{F}}$ associated to $G$ (for the $SL(2,\R)$ case, see Question \ref{qu:hypplane}) \cite{HumeMackTess-PprofLie}.
    \item The Diestel-Leader graphs $DL(m,n)$ with $m,n\geq 2$ all satisfy
    \[
        \sep_{DL(m,n)}(r)\simeq r/\log(r)
    \]
    \cite{HumeMackTess-PprofLie}
\end{itemize}
 
In particularly nice cases, such as all of the examples above, Theorem \ref{sepboundscw} can be upgraded to the following:

\begin{corollary}\label{cwpcalc}
    Let $X$ be a bounded degree graph such that $\sep_X(r)\simeq r^a\log(r)^b$ for some $a\in (0,1]$ and $b\in\R$. Then, for every $p\in[1,+\infty]$,
    \[
        \cw^p_X(r) \simeq r^{a+\frac1p}\log(r)^b
    \]
    In particular, $\sep_X\simeq\tw_X\simeq\pw_X\simeq\cw_X\simeq r^a\log(r)^b$.
\end{corollary}
As a result, we can calculate all the invariants mentioned above in these situations.

\subsection{Bounds on bandwidth}

For bandwidth profiles we know much less. One simple lower bound comes from the growth rate.

\begin{definition}
    Let $X$ be a bounded degree graph. The cogrowth function $\kappa_X:\N\to\N$ is given by 
    \[
        \kappa_X(r)=\min\setcon{k}{\exists A\leq X,\ |VA|= r,\ \diam(A)= k}
    \]
    where $\diam(A)$ is measured using the shortest path metric on $A$ (not using the shortest path metric on $X$).
\end{definition}

\begin{proposition}\label{prop:bwgrowth}
    Let $X$ be a bounded degree graph. Then
    \[
     \bw_X(r) \geq \frac{r-1}{\kappa_X(r)}
    \]
\end{proposition}
In particular, $\bw_X(r)\gtrsim r/\kappa_X(r)$.

We can also find na\"ive upper bounds on $p$-bandwidth

\begin{proposition}\label{prop:boundpbw}
    Let $X$ be a bounded degree graph. Then
    \[
     \bw_X(r)\leq \bw^p_X(r) \lesssim_{\Delta_X,p} r^{\frac1p}\bw_X(r)
    \]
\end{proposition}
Unlike the result for cutwidth, the upper bound is very far from sharp. In particular, for the infinite $3$-regular tree $T$
\[
    \bw^1_T(r)=\cw^1_T(r)\lesssim r\log(r)
\]
while the proposition above only yields an upper bound of $r^2/\log(r)$. \medskip

We view the invariants $\bw^p$ as interpolating between the bandwith $\bw^\infty$ and $\bw^1=\cw^1$ which is determined (up to a possible logarithmic error) by the cutwidth.
\medskip

More interestingly, from \cite[Theorem 5]{BPTW} we immediately deduce

\begin{proposition}\label{prop:bwsep}
    Let $X$ be a bounded degree graph. Then,
    \[
    \bw_X(r)\lesssim r/\log(r/\sep_X(r))
    \]
    In particular, if $\sep_X(r)\lesssim r^a\log(r)^b$ for some $a\in[0,1]$ and $b\in\R$, then
    \[
    \bw_X(r)\lesssim \left\{ \begin{array}{ll}
        r/\log(r) & a<1 \\
        r/\log\log(r) & a=1,\ b<0
    \end{array}\right.
    \]
\end{proposition}
A particularly interesting application of this result is the following immediate combination of Propositions \ref{prop:bwgrowth} and \ref{prop:bwsep}.

\begin{corollary}\label{cor:bwrlogr}
    Let $X$ be a bounded degree graph with exponential growth that satisfies $\sep_X(r)\lesssim r^a$ for some $a<1$. Then
    \[
    \bw_X(r)\simeq r/\log(r)
    \]
\end{corollary}
Some examples of graphs satisfying the above hypotheses include minor-free bounded degree graphs\footnote{specifically the disjoint union of all finite graphs of maximum degree $d$ with no $H$-minor is an example whenever $d\geq 3$ and $H$ is not a forest} \cite{Alon-Seymour-Thomas}, and Cayley graphs of finitely generated groups which are
\begin{itemize}
    \item non-elementary hyperbolic,
    \item relatively hyperbolic with virtually nilpotent peripheral subgroups,
    \item $H\times N$ where $H$ is non-elementary hyperbolic and $N$ is virtually nilpotent
\end{itemize}
and their exponential growth finitely generated subgroups. \cite{HumeMackTess-Pprof,HumeMackTess-PprofLie}

\subsection{Recognising graphs containing expanders}

The separation profile detects whether a bounded degree graph admits a family of expander graphs. This fact was noted in \cite{HumSepExp} as a consequence of a theorem proving a $\simeq$-equivalent definition of the separation profile in terms of the Cheeger constant, but was actually known prior to this. We first recall this definition of an expander.

\begin{definition}
    A family of finite graphs $(\Gamma_n)_{n\in\N}$ is called an \textbf{expander} if $|V\Gamma_n|\to\infty$ as $n\to \infty$ and $\inf_n h(\Gamma_n)>0$, where 
    \[
        h(\Gamma)=\min\setcon{\frac{|\partial A|}{|A|}}{A\subset V\Gamma,\ |A|\leq \frac{|V\Gamma|}{2}}
    \]
    and $\partial A=\setcon{v\in V\Gamma}{d_\Gamma(v,A)=1}$.
\end{definition}

Here we note the following generalisation of \cite[Theorem 1.3]{HumSepExp}, which is proved by combining \cite[Theorem 8]{BPTW}, \cite[Proposition 2.2]{HumSepExp} and the comparisons between cutsize, cutwidth and bandwidth discussed above.

\begin{theorem}
    Let $X$ be a bounded degree graph. The following are equivalent:
    \begin{enumerate}
        \item there is a family of subgraphs $(\Gamma_n)_{n\in\N}$ of $X$ which is an expander,
        \item $\limsup_{r\to\infty} \frac{1}{r}\sep_X(r) >0$,
        \item $\limsup_{r\to\infty} \frac{1}{r}\cw_X(r) >0$,
        \item $\limsup_{r\to\infty} \frac{1}{r}\bw_X(r) >0$.
    \end{enumerate}
\end{theorem}

\subsection{Graph width bounds and coarse wirings}

We note that the $p=1$ case is almost useless as an invariant.

\begin{lemma}\label{lem:wid1useless} For any graphs $G$ and $X$,
    \[
        \wid^{G,1}_X(r) =\min\{r,|VX|\}
    \]
\end{lemma}

Also, the $p=\infty$ case is only of interest for ``badly-connected'' graphs $G$:

\begin{proposition}\label{widinftygrids} Let $G$ be a graph which contains, for each $m\in\N$, a topological subdivision of the $m\times m$ square grid. Then for every graph $X$,
    \[
        \wid^{G,\infty}_X(r) \simeq_{\Delta_X} 1.
    \]    
\end{proposition}

For intermediate values of $p$, we believe that the invariants are much more interesting (cf.\ Questions \ref{qu:bwptree} and \ref{qu:bwpgrid}). Inspired by Proposition \ref{widinftygrids}, we define the following more sensitive version of $1$-width.

\begin{definition}
    Let $\Gamma$ and $G$ be graphs. A \textbf{$k$-slim} $G$-decomposition of $\Gamma$ is a $G$-decomposition where $|X_g|\leq k$ for all $g\in VG$. We define
    \[
    {}^k\wid^{G,1}(\Gamma) = \min \sum_{g\in VG} |X_g|^p
    \]
    where the minimum is taken over all $k$-slim $G$ decompositions. If no such decomposition exists, we write ${}^k\wid^{G,1}(\Gamma)=+\infty$. Finally, for any graph $X$, we define
    \[
        {}^k\wid^{G,1}_X(r)=\max\setcon{{}^k\wid^{G,1}(\Gamma)}{\Gamma\leq X,\ |V\Gamma|\leq r}
    \]
\end{definition}
We could also define ${}^k\wid^{G,p}_X(r)$ for all $p\in[1,\infty)$, but it is always the case that
\[
{}^k\wid^{G,p}_X(r) \simeq_{k,p} ({}^k\wid^{G,1}_X(r))^{1/p}
\]

Note that the following are obviously equivalent:
\begin{itemize}
    \item for every $k$, there is an $r$ such that ${}^k\wid^{G,1}_X(r)=+\infty$,
    \item $\wid^{G,\infty}_X(r)\not\simeq 1$
\end{itemize}
So, for every $G$ and $X$, at least one of ${}^k\wid^{G,1}_X(r)$ and $\wid^{G,\infty}_X(r)$ is non-trivial. Furthermore, the two invariants above have equivalent descriptions in terms of coarse wirings, as introduced in \cite{BarrettHume}. To avoid more lengthy definitions in the introduction, let us say roughly that a $k$-coarse wiring is a continuous map $f:\Gamma\to G$ which maps vertices to vertices, paths to walks, and which ``uses'' each vertex and edge no more than $k$-times (precise definitions are given in \S\ref{sec:wiring}). Given two graphs $X,G$, $\wir^k_{X\to G}(r)$ is the minimal $m$ such that every $r$-vertex subgraph $\Gamma$ of $X$ admits a $k$-coarse wiring into $G$ whose image is contained in a subgraph of $G$ with $m$ vertices. We also define $\para_{X\to G}(r)$ to be the minimal $\ell$ such that every $r$-vertex subgraph of $X$ admits an $\ell$-coarse wiring into $G$. These invariants are connected to graph width as follows:

\begin{proposition}\label{prop:widinftywiring}
    Let $X$ and $G$ be graphs. Then for every $r\in\N$
    \[
        \wid^{G,\infty}_X(r) \leq 2 \para_{X\to G}(r)
    \]
    If, in addition, $X$ has bounded degree, then for every $r\in\N$
    \[
        \para_{X\to G}(r) \leq \max\{1,\Delta_X\}\wid^{G,\infty}_X(r)
    \]
\end{proposition}

\begin{proposition}\label{prop:wid1wiring}
    For every $k$, every finite graph $\Gamma$ and every graph $G$,
    \begin{itemize}
        \item if there is a coarse $k$-wiring $f:\Gamma\to G$ with volume at most $V$, then ${}^{2k}\wid^{G,1}(\Gamma)\leq 2kV$,
        \item if ${}^k\wid^{G,1}(\Gamma)\leq V$ then there is a coarse $\Delta_\Gamma k$-wiring $\Gamma\to G$ with volume at most $(1+\Delta_\Gamma)V$.
    \end{itemize}
\end{proposition}
Coarse wiring volumes have been computed in the case where $X_d$ is the disjoint union of all finite graphs with maximum degree $d\geq 3$, and $G$ is the Cayley graph of a lattice\footnote{a discrete subgroup of the automorphism group with compact quotient} in either a Euclidean space, a real hyperbolic space or a symmetric space whose noncompact factor has rank at least 2.

\begin{corollary}
    If $G$ is a lattice in $\R^n$ or $\HH_\R^{n+1}$ for $n\geq 2$ then for all $k$ sufficiently large
    \[
    {}^k\wid^{G,1}_{X_d}(r) \simeq r^{1+\frac1n}
    \]
    If $G$ is a lattice in a symmetric space whose noncompact factor has rank at least 2, then for all $k$ sufficiently large
    \[
    {}^k\wid^{G,1}_{X_d}(r) \simeq r\log(r)
    \]
\end{corollary}
The case of a lattice in the hyperbolic plane is open, the known bounds are
\[
    \exp(r^\frac12) \lesssim {}^k\wid^{G,1}_{X_d}(r)\lesssim \exp(r)
\]
for all sufficiently large $k$. Proving $\cw_G(r)\simeq \log(r)$ (see Question \ref{qu:hypplane}) would immediately imply that ${}^k\wid^{G,1}_{X_d} (r)\simeq \exp(r)$.

Finally, we note that for certain examples of graphs $X$ and $G$, $\wir^k_{X\to G}$ changes $\simeq$ equivalence class infinitely many times as $k$ increases \cite{Raistrick-wiring}. Therefore, the same also applies to ${}^k\wid^{G,1}_{X}(r)$.

\subsection{Acknowledgements}
This research was conducted during various summer research programmes held at the University of Bristol under the supervision of the second author. The first author was participating in a STEM Research Project from Global Summer, the third and fourth authors were participating in Summer Research Projects funded by the University of Bristol. The fourth author also gratefully acknowledges the support of the Bei Shan Tang Foundation through the Lee Hysan Memorial Scholarship for Overseas Studies.

The second author was supported by the EPSRC grant EP/V027360/1 ``Coarse geometry of groups and spaces''.

The authors would like to thank David Ellis for pointing out several references, Alice Kerr for bringing the reference \cite{Alon-Seymour-Thomas} to their attention, and Romain Tessera for helpful comments on a draft version of the paper.

Finally, the authors would like to thank Jesse Geneson for pointing out two small errors in the published version of Proposition \ref{prop:widinftywiring}, namely that $\wid$ and $\para$ have been reversed in the statement, and that the multiplicative error in the second inequality should be $\max\{1,\Delta_X\}$ rather than $\Delta_X$.

\section{Monotonicity proofs}\label{sec:invariants}

The goal of this section is to prove the three monotonicity results: Theorems \ref{thm:monotoneLA}, \ref{thm:monotonedecomp} and Proposition \ref{prop:widthGinv}. Each utilises a similar construction so we begin with this.

\subsection{Constructing comparison graphs}\label{sec:buildGamma'}

Let $X,Y$ be graphs and let $\phi:VX\to VY$ be a $\kappa$-regular map. For each pair $y,y\in VY$ with $d_Y(y,y')\leq\kappa$ choose a minimal length path $\gamma_{y,y'}$ from $y$ to $y'$ in $Y$. Given a finite subgraph $\Gamma\leq X$, we define the following finite subgraph of $Y$
\[
\Gamma^\phi = \phi(V\Gamma)\cup \bigcup_{xx'\in E\Gamma} \gamma_{\phi(x),\phi(x')}
\]
This procedure has the following notable qualities:
\begin{itemize}
    \item $\frac{1}{\kappa}|V\Gamma|\leq |V\Gamma^\phi| \leq |V\Gamma|+(\kappa-1)|E\Gamma| \leq \left(1+\frac{\kappa-1}{2}\Delta_\Gamma\right)|V\Gamma|$,
    \item if $\Delta_Y<+\infty$, then for every $y\in V\Gamma^\phi$ there are at most $\kappa\Delta_\Gamma (1+\Delta_Y)^\kappa$ edges $xx'\in E\Gamma$ such that $y\in\gamma_{\phi(x),\phi(x')}$.
\end{itemize}
The first of these is straightforward. For the second, note that if $y\in\gamma_{\phi(x),\phi(x')}$ then $d_Y(y,\phi(x))\leq\kappa$. Hence, $x\in \phi^{-1}(B_Y(y,\kappa))$ where $B_Y(y,\kappa)$ the closed ball of radius $\kappa$ centred at $y$. As $\phi$ is $\kappa$-regular, and $Y$ has bounded degree, there are at most
 $\kappa (1+\Delta_Y)^\kappa$ vertices $x\in VX$ such that $y\in\gamma_{\phi(x),\phi(x')}$ for some $x'$ and therefore at most $\kappa \Delta_X(1+\Delta_Y)^\kappa$ edges $xx'\in EX$ such that $y\in\gamma_{\phi(x),\phi(x')}$.
\medskip

In each of the following proofs, we will show that for any of the aforementioned measures of width, the width of $\Gamma^\phi$ cannot be significantly smaller than the width of $\Gamma$.

\subsection{Monotonicity for $p$-cutwidth and $p$-bandwidth}
We recall the statement of Theorem \ref{thm:monotoneLA} for convenience.

\begin{theorem*} Let $X,Y$ be bounded degree graphs with maximum degrees $\Delta_X,\Delta_Y$ respectively. If there is a $\kappa$-regular map $\phi:VX\to VY$, then 
\[
 \cw_X\lesssim_{\kappa,\Delta_X,\Delta_Y} \cw_Y \quad \textup{and} \quad \bw_X\lesssim_{\kappa,\Delta_X,\Delta_Y} \bw_Y
\]
If, in addition, $EY$ is infinite, then for every $p\in [1,+\infty)$ we have
\[
 \cw_X^p\lesssim_{\kappa,\Delta_X,\Delta_Y,p} \cw_Y^p \quad \textup{and} \quad \bw_X^p\lesssim_{\kappa,\Delta_X,\Delta_Y,p} \bw_Y^p
\]    
\end{theorem*}

Before beginning the proof, we give a quick lemma

\begin{lemma}
    Let $Y$ be a graph with infinitely many edges. Then, for all $p\in[1,\infty)$,
    \[
     \cw^p_Y(r) \gtrsim r^{\frac{1}{p}} \quad \textup{and} \quad \bw^p_Y(r) \gtrsim r^{\frac{1}{p}}
    \]
\end{lemma}
\begin{proof}
    For each $r$, let $\Gamma_r$ be an $r$-vertex subgraph of $Y$ with at least $\lfloor\frac{r}{2}\rfloor$ edges. For any bijection $f:V\Gamma_r\to\{1,\ldots,|V\Gamma_r|\}$, 
\[\|\cw_f\|_1=\|\bw_f\|_1 \geq |E\Gamma_r|\geq \left\lfloor\frac{r}{2}\right\rfloor\]
Moreover, for any function $g:A\to\N$ and any $p\in[1,\infty)$
\[
\|g\|_p = \left(\sum_{a\in A} g(a)^p\right)^\frac{1}{p} \geq \left(\sum_{a\in A} g(a)\right)^\frac{1}{p} = (\|g\|_1)^\frac1p
\]
Thus, $\cw^p_Y\gtrsim r^{\frac1p}$ and $\bw^p_Y\gtrsim r^{\frac1p}$ as required.
\end{proof}

\begin{proof}[Proof of Theorem \ref{thm:monotoneLA}]
Let $\Gamma$ be a finite subgraph of $X$, and let $\Gamma^\phi$ be the corresponding subgraph of $Y$.

Given a bijection $g:V\Gamma^\phi\to \{1,\ldots,|V\Gamma^\phi|\}$, choose a bijection $f:V\Gamma\to\{1,\ldots,|V\Gamma|\}$ with the property that $f(v)\leq f(v')$ implies $g(\phi(v))\leq g(\phi(v'))$. In particular, this means that the sequence $a_i=g(\phi(f^{-1}(i)))$ is monotonically non-decreasing.
\medskip 

\noindent\textbf{Bandwidth:} We claim that for every edge $vv'\in E\Gamma$ such that $|f(v)-f(v')|> \kappa^2$, there is some edge $ww'\in E\Gamma^\phi$ satisfying
\begin{itemize}
	\item $d_Y(\phi(v),w)\leq \kappa$, and
	\item $|f(v)-f(v')| \leq \kappa^2 |g(w)-g(w')|$.
\end{itemize}
Given the claim we complete the proof as follows. By the first condition above there is a uniform bound (call it $C$) on the number of edges $vv'$ which yield the same edge $ww'$.
Thus for $p\in[1,\infty)$,
\begin{align*}
 \bw^p(\Gamma) 	& = \min_f\left(\sum_{vv'\in E\Gamma} |f(v)-f(v')|^p\right)^{1/p}
 		\\
 				& \leq \min_g\left(C\sum_{ww'\in E\Gamma^\phi} (\kappa^2|g(w)-g(w')|)^p\right)^{1/p} + \left(\sum_{vv'\in E\Gamma} (\kappa^2)^p\right)^{1/p}
 		\\
 				& \leq \kappa^2\left(C^{1/p} \bw^p(\Gamma^\phi) + |E\Gamma|^{1/p}\right)
 		\\
 				& \lesssim \bw^p_Y(\max\{|V\Gamma^\phi|,|E\Gamma|\})
            \\
                    & \lesssim \bw^p_Y(|\Gamma|)
\end{align*}
Note that the penultimate step uses the assumption that $\bw^p_Y(r)\gtrsim r^{1/p}$, and the last step uses the facts $|V\Gamma^\phi|\leq \left(1+\frac{\kappa-1}{2}\Delta_X\right)|V\Gamma|$ and $|E\Gamma|\leq \frac{\Delta_X}{2}|V\Gamma|$. For $p=\infty$,
\begin{align*}
 \bw^\infty(\Gamma) 	& = \min_f\max_{vv'\in E\Gamma} |f(v)-f(v')|
 		\\
 				& \leq \min_g\max_{ww'\in E\Gamma^\phi} \kappa^2|g(w)-g(w')| + 2\kappa^2
 		\\
 				& \leq \kappa^2\left(\bw^\infty(\Gamma^\phi)+1\right).
            \\
                    & \lesssim \bw^\infty_Y(|V\Gamma|)
\end{align*}

It remains to prove the claim. Suppose $|f(v)-f(v')|=\ell> \kappa^2$. Without loss of generality, assume $k=f(v)<f(v')=k+\ell$. Since $f$ and $g$ are injective, we have
\[
 \abs{g(\phi(\{f^{-1}(k),\ldots,f^{-1}(k+\ell)\}))} \geq \ell/\kappa
\]
and $g(\phi(f^{-1}(k)))\leq \ldots \leq g(\phi(f^{-1}(k+\ell)))$, so $g(\phi(f^{-1}(k+\ell)))-g(\phi(f^{-1}(k)))\geq \ell/\kappa$. In particular, $\phi(v)\neq\phi(v')$.

Since $vv'\in E\Gamma$, $d_Y(\phi(v),\phi(v'))\leq \kappa$. Let $\phi(v)=w_0,w_1,\ldots,w_m=\phi(v')$ be the vertices of a path in $\Gamma^\phi$ with $m\leq \kappa$. Since $g(w_m)-g(w_0)\geq \ell/\kappa$, there is some $i$ such that $|g(w_{i+1})-g(w_i)|\geq \ell/(\kappa^2)$. We choose $w=w_i$ and $w'=w_{i+1}$. By construction $ww'\in E\Gamma^\phi$, $d_Y(\phi(v),w)\leq \kappa$ and $|f(v)-f(v')|\leq \kappa^2|g(w)-g(w')|$, completing the claim.
\medskip

\noindent\textbf{Cutwidth:} We use the same bijection $f:V\Gamma\to\{1,\ldots,|V\Gamma|\}$. Fix $1<k\leq |V\Gamma|$. For convenience define
\begin{align*}
F_k & = \setcon{vv'\in E\Gamma}{f(v)\leq k< f(v')} \\
G_\ell & = \setcon{ww'\in E\Gamma^\phi}{g(w)\leq \ell< g(w')}
\end{align*}
and write $f_k=|F_k|$, $g_\ell=|G_\ell|$. We claim that $f_{k}\leq Cg_{a_k}+\kappa\Delta_X$ for some constant $C$ which depends only on $\kappa$ and $\Delta_Y$. We recall that $a_k=g(\phi(f^{-1}(k)))$.

Given the claim, we see that for $p=\infty$,
\begin{align*}
 \cw^\infty(\Gamma) 	& = \min_f\max_k f_k
 		\\
 				& \leq \min_g\max_{\ell} Cg_{\ell}+\kappa\Delta_X 
 		\\
 				& \leq C \cw^\infty(\Gamma^\phi) + \kappa\Delta_X.
\end{align*}
When $p\in[1,\infty)$, using the fact that for any $\ell$ there are at most $\kappa$ values of $k$ such that $a_k=\ell$, we have
\begin{align*}
 \cw^p(\Gamma) 	& = \min_f\left(\sum_k f_k^p\right)^{1/p}
 		\\
 				& \leq \min_g\left(\sum_\ell\kappa\left(Cg_\ell+\kappa\Delta_X \right)^p\right)^{1/p} 
 		\\
 				& \leq C\kappa^{1/p}\min_g\left(\sum_\ell\left(g_\ell\right)^p\right)^{1/p} + C\kappa^{1+1/p}\Delta_X |V\Gamma^\phi|^{1/p}
  		\\
 				& \leq C\kappa^{1/p}\max\left\{\min_g\left(\sum_l\left(g_l\right)^p\right)^{1/p}, \kappa\Delta_X |V\Gamma^\phi|^{1/p}\right\}
 		\\
 				& \lesssim \cw^p_Y(|V\Gamma^\phi|).
    	\\
 				& \lesssim \cw^p_Y(|V\Gamma|).
\end{align*}
Note that the final step uses the assumption made at the start of the proof that $\cw^p_Y(r) \gtrsim r^{1/p}$.

Now we prove the claim. For each edge in the set $F_k$, we have $g(\phi(v))\leq a_k \leq g(\phi(v'))$. Since $\phi$ has preimages of cardinality at most $\kappa$ there are at most $\kappa\Delta_X $ edges in $E\Gamma$ with an end vertex $w$ such that $g(\phi(w))=a_k$.

If $vv'\in F_k$ is not one of these edges, then $g(\phi(v))\leq a_k < g(\phi(v'))$. There is a path in $\Gamma^\phi$ connecting $\phi(v)$ to $\phi(v')$ with length at most $\kappa$, and there is some edge $ww'$ on this path such that $g(w)\leq a_k < g(w')$.

As in the previous argument, since $d_Y(\phi(v),w))\leq \kappa$ there is a uniform bound $C$ on the number of edges $vv'$ which yield the same edge $ww'$ via this process. Thus,
\[
 f_k \leq Cg_{a_k}+\kappa\Delta_X,
\]
as required.
\end{proof}

\subsection{Monotonicity for $p$-graph width}

We first prove monotonicity with respect to the graphs being decomposed (Theorem \ref{thm:monotonedecomp}).

\begin{theorem*} Let $X,Y$ be bounded degree graphs with maximum degrees $\Delta_X,\Delta_Y$ respectively. If there is a regular map $\phi:VX\to VY$, then for every graph $G$ and every $p\in[1,\infty]$ we have
\[
 \wid^{G,p}_X\lesssim_{\Delta_X,\Delta_Y,G,p} \wid^{G,p}_Y
\]
\end{theorem*}
\begin{proof}[Proof of Theorem \ref{thm:monotonedecomp}]
    Let $\Gamma$ be a finite subgraph of $X$ and let $\Gamma^\phi\leq Y$ be the finite subgraph of $Y$ constructed in $\S\ref{sec:buildGamma'}$. Let $\{X'_g\}_{g\in G}$ be a $G$-decomposition of $\Gamma^\phi$ and, for each $g\in G$ let
    \[
    X_g=\setcon{x\in V\Gamma}{X'_g\cap V\overline\phi(x)\neq\emptyset}
    \]
    where $\overline\phi(x)=\{x\}\cup \bigcup_{xx'\in E\Gamma} \gamma_{\phi(x),\phi(x')}$.
\medskip

\noindent    \textbf{Step 1: $\{X_g\}_{g\in VG}$ is a $G$-decomposition of $\Gamma$} \smallskip

    We recall that, as $\{X'_g\}_{g\in G}$ is a $G$-decomposition of $\Gamma^\phi$:
    \begin{itemize}
        \item for every $y\in V\Gamma^\phi$, there is some $g\in G$ such that $y\in X'_g$;
        \item for every edge $yy'\in E\Gamma^\phi$, there is some $g\in G$ such that $y,y'\in X'_g$;
        \item for each $y\in V\Gamma^\phi$, the induced subgraph of $G$ with vertex set $G'_y:=\setcon{g\in VG}{y\in X'_g}$ is connected.
    \end{itemize}
    For each $x\in V\Gamma$, $\phi(x)\in X'_g$ for some $g$, so $x\in X_g$. If $xx'\in E\Gamma$, then $\phi(x),\phi(x')\in \overline\phi(x)$, so $x'\in X_g$ whenever $x\in X_g$.

By construction, for each $x\in V\Gamma$,
\begin{equation}\label{eq:pots}
	\setcon{g\in VG}{x\in X_g} = \bigcup_{y\in V\overline{\phi}(x)} \setcon{g\in VG}{y\in X'_g}.
\end{equation}
Each $\setcon{g\in VG}{y\in X'_g}$ is connected by assumption. Moreover, whenever $yy'\in E\Gamma^\phi$, there is some $g$ such that $y,y'\in X'_g$, hence
\[
\setcon{g\in VG}{y\in X'_g} \cap \setcon{g\in VG}{y'\in X'_g}\neq\emptyset
\]
As a result, $\bigcup_{y\in \overline{\phi}(x)} \setcon{g\in VG}{y\in X'_g}$ induces a connected subgraph of $G$, since $\overline{\phi}(x)$ is connected.
\medskip

\noindent\textbf{Step 2: bounding $\wid^{G,p}(\Gamma^\phi)$}\smallskip

For a fixed $y\in V\Gamma^\phi$, if $y\in\overline{\phi}(x)$ then $\phi(x)$ is contained in the ball of radius $2\kappa$ centred at $y$. As $|\phi^{-1}(z)|\leq \kappa$ for every $z\in VY$, we deduce that for each $x\in V\Gamma$,
\[
|\phi^{-1}(\overline\phi(x))|\leq\kappa(1+\Delta_X^{2\kappa})
\]
Thus, for each $g\in G$,
\[
 |X_g| \leq \kappa(1+\Delta_X^{2\kappa})|X'_g|
\]
Therefore, 
\[
\wid^{G,p}(\Gamma) \leq \kappa(1+\Delta_X )^{2\kappa}\wid^{G,p}(\Gamma^\phi).
\]
As this holds for all finite $\Gamma\leq X$, we have
$\wid^{G,p}_X(r) \lesssim \wid^{G,p}_Y(r)$.
\end{proof}

Next, we prove monotonicity for the graphs we are decomposing over (Proposition \ref{prop:widthGinv}).

\begin{proposition*} Let $G,G'$ be bounded degree graphs. If there is a $\kappa$-regular map $\phi:VG\to VG'$, then for every $p\in[1,+\infty]$ and every bounded degree graph $X$
    \[
        \wid^{G',p}_X(r) \lesssim_{p,\kappa,\Delta_G,\Delta_{G'}} \wid^{G,p}_X(r)
    \]
\end{proposition*}
\begin{proof}
    Let $\Gamma$ be a finite subgraph of $X$ and let $\{X_g\}_{g\in VG}$ be a $G$-decomposition of $\Gamma$. We will define a $G'$-decomposition $\{X'_{g'}\}_{g'\in VG'}$ of $\Gamma$ by describing the sets $\{G'_x\}_{x\in V\Gamma}$. Specifically, we apply the construction from \S\ref{sec:buildGamma'} to the regular map $\phi:VG\to VG'$ and define $G^\phi_x$ to be the graph obtained from $G_x$ under this construction. It is immediate that each $\Gamma^\phi_x$ is connected and $x\in X'_{\phi(g)}$ for every $g$ satisfying $x\in X_g$. Hence $\bigcup X'_{g'}=V\Gamma$. Finally, $x,y\in X'_{\phi(g)}$ whenever $x,y\in X_g$, so $\{X'_{g'}\}_{g'\in VG'}$ defines a $G'$-decomposition of $\Gamma$.

    Now $X'_{g'}=\bigcup_{\setcon{g}{g'\in\overline\phi(g)}} X_g$. Define a function $f:VG'\to VG$ with the property that $f(g')\in \setcon{g}{g'\in\overline\phi(g)}$ and
    \[
      |X_{f(g')}|=\max_{\setcon{g}{g'\in\overline\phi(g)}} |X_g|
    \]
    for all $g\in G'$. As each $\overline\phi(g)$ has diameter at most $2\kappa$ and $G$ and $G'$ have bounded degree, there is some $C=C(\kappa,\Delta_G,\Delta_{G'})$ such that 
    \begin{itemize}
        \item $|f^{-1}(g)|\leq C$ for all $g\in G$, and
        \item $|\setcon{g}{g'\in\overline\phi(g)}|\leq C$ for all $g'\in VG'$.
    \end{itemize}
    Hence, for $p=\infty$, we have
    \[
     \max_{g'\in VG'} |X'_{g'}| \leq C\max_{g\in VG} |X_{g}|
    \]
    while for $p\in[1,\infty)$ we have
    \begin{align*}
        \sum_{g'\in VG'}|X'_{g'}|^p & \leq \sum_{g'\in VG'} \left( \sum_{\setcon{g}{g'\in\overline\phi(g)}} |X_g| \right)^p \\
        & \leq C\sum_{g\in VG} (C|X_g|)^p \\
        & = C^{p+1} \sum_{g\in VG} |X_g|^p 
    \end{align*}
    As this holds for all finite subgraphs $\Gamma\leq X$ and all $G$-decompositions of $\Gamma$, we deduce that for all $p\in[1,\infty]$
    \[
    \wid^{G',p}_X(r) \lesssim_{p,\kappa,\Delta_G,\Delta_{G'}} \wid^{G,p}_X(r) \qedhere
    \]
\end{proof}

\section{Comparing and calculating invariants}

\subsection{Comparing new invariants}
As mentioned above, the proof of Proposition \ref{prop:compare} is already in the literature. Here we list a family of references from which the Proposition immediately follows.
\medskip

\begin{proof}[Proof of Proposition \ref{prop:compare}]
    Firstly, for any finite graph $\Gamma$, $\cut(\Gamma)\leq \tw(\Gamma)$ \cite[Lemma 2.3]{BenSchTim-12-separation-graphs}, so for any graph $X$
    \begin{align*}
    \sep_X(r)=&\max\setcon{\cut(\Gamma)}{\Gamma\leq X,\ |V\Gamma|\leq r} \\ \leq & \max\setcon{\tw(\Gamma)}{\Gamma\leq X,\ |V\Gamma|\leq r} = \tw_X(r)
    \end{align*}
    As every path is a tree, for any graph $\Gamma$, $\tw(\Gamma)\leq \pw(\Gamma)$, so
    \[
    \tw_X(r)\lesssim \pw_X(r)
    \]
    For any finite graph $\pw(\Gamma) \leq \cw(\Gamma)$ by \cite[Theorem 5.4]{Bodlaender_Classes} and $\cw(\Gamma)\leq \Delta_\Gamma\pw(\Gamma)$ by \cite{Chung-Seymour}, so for every bounded degree graph $X$:
    \[
     \pw_X(r)\simeq \cw_X(r).
    \]
    Finally, using \cite[Theorem 3.2]{MS89} and \cite[Corollary 2.2]{MPS83} we have
    \[
    \cw(\Gamma) \leq\left\lfloor \frac{\Delta_\Gamma}2\right\rfloor (\sn(\Gamma)-1)+1 \leq \left\lfloor \frac{\Delta_\Gamma}2\right\rfloor \tbw(\Gamma)+1\leq \left\lfloor \frac{\Delta_\Gamma}2\right\rfloor \bw(\Gamma)+1
    \]
    where $\sn(\Gamma)$ is the \textbf{search number}\footnote{We will not define search number in this paper as it is not needed elsewhere.} of $\Gamma$ and $\tbw(\Gamma)$ -- the \textbf{topological bandwidth} of $\Gamma$ -- is the minimal bandwidth over all subdivisions of $\Gamma$. It is then immediate from the definition that $\tbw(\Gamma)\leq \bw(\Gamma)$.  Hence, for any bounded degree graph $X$,
    \[
    \cw_X(r)\lesssim \bw_X(r) \qedhere
    \]
\end{proof}

\begin{remark}\label{rem:otherinvs}
    As $\tbw(\Gamma) \leq \cw(\Gamma)$ for any finite graph $\Gamma$ \cite[Corollary 2.1]{MPS83}, one additional consequence of this result is that we could also define coarse graph invariants based on topological bandwidth ($\tbw$) and search number $(\sn)$:
    \begin{align*}
        \tbw_X(r) & = \max\setcon{\tbw(\Gamma)}{\Gamma\leq X,\ |V\Gamma|\leq r} \\
        \sn_X(r) & = \max\setcon{\sn(\Gamma)}{\Gamma\leq X,\ |V\Gamma|\leq r}
    \end{align*}
    For any bounded degree graph $X$ we have
    \[
    \cw_X \simeq \tbw_X \simeq \sn_X
    \]
    Additionally, for each $p\in[1,+\infty]$ we may define the $p$-\textbf{vertex separation} of a finite graph $\Gamma$ (denoted $\vs^p(\Gamma)$) as the minimum over all orderings $f$ of $\|\vs_f\|_p$ where $\vs:\{1\ldots,|V\Gamma|\}\to\N$ is defined by
    \[
     \vs_f(i)=\abs{\setcon{x\in V\Gamma}{\exists xy\in E\Gamma,\ f(x)\leq i< f(y)}}
    \]
    Here, $\vs^\infty$ is the typical definition of vertex separation, and $\vs^1$ is \textbf{sumcut} which is also equal to \textbf{profile}. It is immediate that for any ordering $f$ and any $i$
    \begin{align*}
    \frac{1}{\Delta_\Gamma}\vs_f(i) & \leq \cw_f(i):=\max_i \abs{\setcon{x\in V\Gamma}{\exists xy\in E\Gamma,\ f(x)\leq i< f(y)}} \\ & \leq \vs_f(i)
    \end{align*}
    Hence, defining $\vs^p_X(r)=\max\setcon{\vs^p(\Gamma)}{\Gamma\leq X,\ |V\Gamma|\leq r}$, we have
    \[
     \vs_X^p\simeq \cw^p_X
    \]
    for every bounded degree graph $X$ and every $p\in[1,+\infty]$.
    
    Thus, whenever there is a $\kappa$-regular map $\phi:VX\to VY$ where $X,Y$ are bounded degree graphs,
    \begin{align*}
        \tbw_X & \lesssim_{\kappa,\Delta_X,\Delta_Y} \tbw_Y  \\ \sn_X &\lesssim_{\kappa,\Delta_X,\Delta_Y} \sn_Y  \\
        \vs^p_X &\lesssim_{\kappa,\Delta_X,\Delta_Y,p} \vs^p_Y
    \end{align*}
\end{remark}

\subsection{Comparisons with the separation profile}
Our goal in this section is to prove Theorem \ref{sepboundscw} and Corollary \ref{cwpcalc}.

\begin{theorem*}
 Let $X$ be a graph with maximum degree $\Delta_X$. Then, for every $p\in[1,+\infty]$
    \[
     r^{\frac1p}\sep_X(r) \lesssim_{\Delta_X,p} \cw^p_X(r) \lesssim r^{\frac1p}\cw_X(r) \lesssim_{\Delta_X,p} r^{\frac1p}\sep_X(r) (1+\log_2(r))
    \]
\end{theorem*}

Firstly, for any finite graph $\Gamma$, any $p\in[1,\infty)$ and any bijection $f:V\Gamma\to\{1,\dots,|V\Gamma|\}$
\[
    \left(\sum_{i=1}^{|V\Gamma|} \cw_f(i)^p\right)^{\frac1p} \leq \left(\sum_{i=1}^{|V\Gamma|} \max_i\{\cw_f(i)\}^p\right)^{\frac1p} \leq |V\Gamma|^{\frac1p} \max_i\{\cw_f(i)\}
\]
so for any graph $X$ and any $p\in[1,\infty)$
\[
\cw^p_X(r) \leq r^{\frac1p}\cw_X(r)
\]
for all $r\in\N$. \medskip

The upper bound on $\cw_X(r)$ is a consequence of the following divide-and-conquer lemma.

\begin{lemma}
    Let $X$ be a bounded degree graph. For every $r$,
    \[
     \cw_X(r) \leq \cw_X(r/2) + \Delta_X\sep_X(r)
    \]
\end{lemma}
\begin{proof}
    Let $\Gamma$ be a finite subgraph of $X$ with at most $r$ vertices and choose $C\subseteq V\Gamma$ such that $|C|\leq\sep_X(r)$ and every connected component of $\Gamma-C$ has at most $r/2$ vertices. This is possible by the definition of $\sep_X(r)$. Now, let $A_1,\ldots,A_k$ be the connected components of $\Gamma-C$. For each connected component, choose an ordering $f_\ell:\{1,\ldots,|VA_\ell|\}\to\N$ such that
    \[
     \|\cw_{f_\ell}\|_\infty = \cw^\infty(A_\ell) \leq \cw^\infty_X(r/2)
    \]
    Now define $f:V\Gamma\to\{1,\ldots,|V\Gamma|\}$ by
    \[ 
        f(v) = \left\{
                        \begin{array}{ll}
                            \sum_{j=1}^{\ell-1} |VA_j|+f_\ell(v) & v\in A_\ell \smallskip \\
                            \sum_{j=1}^{k} |VA_j| + g(v) & v\in C 
                        \end{array}\right.
    \]
    where $g:C\to \{1,\ldots,|C|\}$ is any bijection. As the $A_\ell$ are the connected components of $\Gamma-C$, every edge in $E\Gamma$ is either in some $EA_\ell$ or has an end vertex in $C$. Hence, if $i\in f^{-1}(A_\ell)$ for some $\ell$, then
    \begin{align*}
     \cw_f(i) \leq & \abs{\setcon{vw\in EA_\ell}{f_\ell(v) \leq i - \sum_{j=1}^{\ell-1} |VA_j| < f_\ell(w)}} \smallskip\\ & + \abs{\setcon{vw\in E\Gamma}{w\in C}} \smallskip \\
      \leq & \|\cw_{f_\ell}\|_\infty + \Delta_X|C|
    \end{align*}
    If $f^{-1}(i)\in C$, then $\cw_f(i)\leq \Delta_X|C|$. Hence,
    \[
    \max_i\{\cw_f(i)\} \leq \max_\ell \{\|\cw_{f_\ell}\|_\infty\}+\Delta_X|C| \leq \cw_X(r/2) + \Delta_X\sep_X(r)
    \]
    As this holds for every subgraph $\Gamma$ of $X$ with at most $r$ vertices,
    \[
    \cw_X(r) \leq \cw_X(r/2) + \Delta_X\sep_X(r) \qedhere
    \]
\end{proof}

Iterating this procedure we obtain the following:

\begin{corollary}\label{cor:cwsep}
    For any bounded degree graph $X$,
    \[
    \cw_X(r) \leq \Delta_X\sum_{\ell=0}^{\lfloor \log_2r\rfloor} \sep_X\left(\frac{r}{2^\ell}\right) \leq \Delta_X(1+\log_2r) \sep_X(r)
    \]
\end{corollary}
\begin{proof}
    We apply Proposition \ref{prop:compare} $\lfloor\log_2r\rfloor$ times to obtain
    \[
    \cw_X(r) \leq \Delta_X\sum_{\ell=0}^{\lfloor \log_2r\rfloor} \sep_X\left(\frac{r}{2^\ell}\right) + \cw_X\left(\frac{r}{2^{\lfloor \log_2r\rfloor}}\right)
    \]
    As $\lfloor \log_2r\rfloor>\log_2r-1$, $\frac{r}{2^{\lfloor \log_2r\rfloor}}<2$, so $\cw_X\left(\frac{r}{2^{\lfloor \log_2r\rfloor}}\right)=0$, as required.
\end{proof}

The remaining step in the proof of Theorem \ref{sepboundscw} is the bound $r^{\frac1p}\sep_X(r) \lesssim_{\Delta_X,p} \cw^p_X(r)$.

\begin{lemma}\label{lem:cwplower}
    Let $\Gamma$ be a finite graph and let $p\in [1,\infty)$. We have
\begin{align*}
 \cw^p(\Gamma) & \geq \left\lfloor\frac{|V\Gamma|}{3}\right\rfloor ^{\frac1p} \cut^{\frac23}(\Gamma).
\end{align*}
Hence, for any graph $X$,
\[
 \cw_X^p(r) \gtrsim r^{1/p}\sep_X(r)
\]
\end{lemma} 
\begin{proof}
Choose an ordering $f:V\Gamma\to\{1,\ldots, |V\Gamma|\}$, and note that for each $\frac{|V\Gamma|}{3}\leq i\leq \frac{2|V\Gamma|}{3}$,
\[
 C_i:=\setcon{v\in V\Gamma}{\exists vw\in E\Gamma,\ f(v)\leq i < f(w)}
\]
is a $\frac23$-cutset of $\Gamma$, since no connected component of $\Gamma\setminus C_i$ can contain a vertex in both $f^{-1}(\{1,\ldots,i\})$ and $f^{-1}(\{i+1,\ldots,|V\Gamma|\})$. Moreover, $|C_i|\leq \cw_f(i)$.
As there are at least $\left\lfloor\frac{|V\Gamma|}{3}\right\rfloor$ such values of $i$, we have
\[
 \left(\sum_{i=1}^{|V\Gamma|} |\cw_f(i)|^p\right)^\frac1p \geq \left\lfloor\frac{|V\Gamma|}{3}\right\rfloor ^{\frac1p} \cut^{\frac23}(\Gamma)
\]
As this holds for every function $f:V\Gamma\to\{1,\ldots,|V\Gamma|\}$,
\[
 \cw^p(\Gamma) \geq \left\lfloor\frac{|V\Gamma|}{3}\right\rfloor ^{\frac1p} \cut^{\frac23}(\Gamma).\qedhere
\]
\end{proof}

Now we move to the proof of Corollary \ref{cwpcalc}. When the separation profile is known to be sufficiently nice, the upper bound in Proposition \ref{prop:compare} can be improved.

\begin{corollary}\label{cor:cwsepspecific} Let $X$ be a bounded degree graph such that $\sep_X(r) \lesssim r^a\log(r)^b$ where $a\in[0,1]$ and $b\in\R$. Then
\[
\cw_X \lesssim \left\{ \begin{array}{ccc}
r^a\log(r)^b & \textup{if} & a>0, \\
\log(r)^{b+1} & \textup{if} & a=0.
\end{array} \right.
\]
\end{corollary}
\begin{proof}
By Corollary \ref{cor:cwsep}, there is some constant $C$ such that
\begin{align*}
 \cw_X(r) & \leq \Delta_X\sum_{\ell=0}^{\lfloor \log_2r\rfloor} \sep_X\left(\frac{r}{2^\ell}\right) \\
 					& \leq \Delta_X\sum_{\ell=0}^{\lfloor\log_2(r)\rfloor} C(C2^{-\ell}r+C)^a\log(C2^{-\ell}r+C)^b \\
 					& \leq \Delta_XC(2C)^a\log(2Cr)^b\sum_{\ell=0}^{\lfloor\log_2(r)\rfloor} (2^{-\ell}r)^a \end{align*}
If $a>0$, then $\sum_{\ell=0}^{\lfloor\log_2(r)\rfloor} (2^{-\ell}r)^a \leq \frac{r^a}{1-2^{-a}}$, so
\[
	\cw_X(r) 	\leq C'r^a \log(2Cr)^b\simeq r^a\log(r)^b.
\]
If $a=0$, then $\sum_{\ell=0}^{\lfloor\log_2(r)\rfloor} (2^{-\ell}r)^a = 1+\lfloor\log_2(r)\rfloor$, so 
\[
	\cw_X(r) 	\leq \Delta_X C\log(2Cr)^b(1+\lfloor\log_2(r)\rfloor) \simeq \log(r)^{b+1}.\qedhere
\]
\end{proof}

Finally, we combine the above results to prove Theorem \ref{cwpcalc}.

\begin{proof}
By Corollary \ref{cor:cwsepspecific}, if $\sep_X(r) \lesssim r^a\log(r)^b$, with $a\in(0,1]$ then for every $p\in[1,\infty)$
\[
 \cw^p_X(r) \lesssim r^{\frac1p}\cw_X(r) \lesssim r^{\frac1p+a}\log(r)^b
\]
If, in addition $\sep_X(r) \gtrsim r^a\log(r)^b$, then by Lemma \ref{lem:cwplower}
\[
 r^{\frac1p+a}\log(r)^b \lesssim r^\frac1p\sep_X(r) \lesssim r^\frac1p\cw_X(r) \lesssim \cw^p_X(r)\lesssim r^{\frac1p+a} \log(r)^b. \qedhere
\]
\end{proof}

\subsection{Bandwidth, cogrowth and separation}

Firstly, we give the elementary proofs of Propositions \ref{prop:bwgrowth} and \ref{prop:boundpbw}

\begin{proposition*}
    Let $X$ be a bounded degree graph. Then
    \[
     \bw_X(r) \geq \frac{r-1}{\kappa_X(r)}
    \]
\end{proposition*}
\begin{proof}
    Fix $r$ and let $\Gamma\leq X$ satisfy $|V\Gamma|=r$ and $\diam(\Gamma)= \kappa_X(r)$. Choose any ordering $f$ of $V\Gamma$ and let $v=f^{-1}(1)$, $w=f^{-1}(r)$. By assumption there is a path of length $\ell\leq \kappa_X(r)$ $v=v_0,v_1,\ldots,v_\ell=w$ from $v$ to $w$ in $\Gamma$, hence
    \[
    r-1 \leq \sum_{i=1}^\ell |f(v_{i-1})-f(v_{i})|
    \]
    so, for some $i$, $|f(v_{i-1})-f(v_{i})|\geq \frac{r-1}{\ell}\geq \frac{r-1}{\kappa_X(r)}$ as required.
\end{proof}

Next, we give an elementary comparison between bandwidth and $p$-bandwidth.

\begin{proposition*}
    Let $X$ be a bounded degree graph. Then
    \[
     \bw_X(r)\leq \bw^p_X(r) \lesssim_{\Delta_X,p} r^{\frac1p}\bw_X(r)
    \]
\end{proposition*}
\begin{proof}
    Let $\Gamma\leq X$ with $|V\Gamma|\leq r$. For any ordering $g$ and any $p\in[1,\infty)$
    \[
        \max_{xy\in E\Gamma}|g(x)-g(y)|\leq \left(\sum_{xy\in E\Gamma} |g(x)-g(y)|^p\right)^\frac1p
    \]
    Now, choose an ordering $f$ of $V\Gamma$ satisfying $\max_{xy\in E\Gamma}|f(x)-f(y)|\leq\bw_X(r)$. Then
    \begin{align*}
     \bw^p(\Gamma) & \leq \left(\sum_{xy\in E\Gamma} |f(x)-f(y)|^p\right)^\frac1p \\
     & \leq |E\Gamma|^\frac{1}{p} \bw^\infty_X(r) \\
     & \leq \left(\frac{\Delta_X}{2}|V\Gamma|\right)^\frac1p\bw^\infty_X(r) \\
     & \lesssim r^\frac1p\bw^\infty_X(r)
    \end{align*}
    As these bounds hold for all $\Gamma\leq X$ with $|V\Gamma|\leq r$,
    \[
        \bw_X(r)\leq \bw^p_X(r) \lesssim_{\Delta_X,p} r^{\frac1p}\bw_X(r)\qedhere
    \]
\end{proof}

We finish this section with the proofs of Proposition \ref{prop:bwsep} and Corollary \ref{cor:bwrlogr}.

\begin{proposition*}
    Let $X$ be a bounded degree graph. Then, for every $r$
    \[
    \bw_X(r)\leq \frac{6r}{\log_{\max\{2,\Delta_X\}}\left(\displaystyle\frac{r}{\sep^{\frac13}_X(r))}\right)}
    \]
    In particular, if $\sep_X(r)\lesssim r^a\log(r)^b$ for some $a\in[0,1]$ and $b\in\R$, then
    \[
    \bw_X(r)\lesssim \left\{ \begin{array}{ll}
        r/\log(r) & a<1 \\
        r/\log\log(r) & a=1,\ b<0
    \end{array}\right.
    \]
\end{proposition*}
\begin{proof}
The key ingredient in this proof is \cite[Theorem 5]{BPTW}, which states that for any finite graph $\Gamma$,
\[
\bw(\Gamma) \leq \frac{6|V\Gamma|}{\log_{\Delta}(|V\Gamma|/s(\Gamma))}
\]
where $\Delta=\max\{2,\Delta_\Gamma\}$ and $s(\Gamma)$ is the minimal cardinality of a subset $S\subseteq V\Gamma$ such that we may choose $A,B\subseteq V\Gamma$ with the following properties:
\begin{itemize}
    \item[(a)] $V = A \sqcup B \sqcup S$,
    \item[(b)] $|A|, |B|\leq \frac23|V\Gamma|$,
    \item[(c)] $E(A,B)=\emptyset$
\end{itemize}
Firstly, we prove that $s(\Gamma)\leq\cut^{\frac13}(\Gamma)$. Let $S$ be a $\frac13$-cutset of $\Gamma$ and let $A_1,\ldots,A_k$ be the vertex sets of connected components of $\Gamma-S$ ordered so that $\frac13|V\Gamma|\geq |A_1|\geq|A_2|\geq\ldots$. Choose $\ell$ minimal so that $|A_1\cup\ldots\cup A_\ell|>\frac13|V\Gamma|$ and set $A=A_1\cup\ldots\cup A_\ell$ and $B=A_{\ell+1}\cup \ldots\cup A_k$. It is immediate that conditions $(a)-(c)$ above are satisfied, so $s(\Gamma)\leq\cut^{\frac13}(\Gamma)$. Hence,
\[
\bw(\Gamma) \leq \frac{6|V\Gamma|}{\log_{\Delta}(|V\Gamma|/\cut^{\frac13}(\Gamma))}
\]
Taking the maximum over all subgraphs of $X$ with at most $r$ vertices, we see that
\[
\bw_X(r)\leq \frac{6r}{\log_{\Delta}\left(\displaystyle\frac{r}{\sep^{\frac13}_X(r))}\right)} \lesssim \frac{r}{\log(r/\sep_X(r))}
\]
Now if $\sep_X(r)\lesssim r^a\log_2(r)^b$, then there is some $C$ (which without loss of generality we assume is at least $2$) such that for every $r\geq 1$
\begin{align*}
    \bw_X(r)    & \leq \frac{6r}{\log_{\Delta}\left(\frac{r}{C(Cr)^a\log_2(Cr)^b+C}\right)}
\end{align*}
If $a=1$, $b<0$, then for all sufficiently large $r$
\begin{align*}
    \log_{\Delta}\left(\frac{r}{C(Cr)^a\log_2(Cr)^b+C}\right)  & \geq \log_{\Delta}\left(\frac{\log_2(r)^{-b}}{2C^2}\right)  \\ &\geq \log_\Delta\log_2(r) + \log_\Delta(-b/2C^2) \\ & \geq \frac{1}{2\log_2(\Delta)}\log_2\log_2(r)
\end{align*}
    Hence $\bw_X(r) \lesssim r/\log\log(r)$.
\medskip

If $a<1$, then for all sufficiently large $r$
\begin{align*}
    \log_{\Delta}\left(\frac{r}{C(Cr)^a\log_2(Cr)^b+C}\right) & \geq \log_{\Delta}\left(\frac{r^{1-a}}{2C^{1+a}\log_2(Cr)^{-b}}\right) \\ & \geq \frac{1-a}{2\log_2(\Delta)}\log_2(r)
\end{align*}
Hence $\bw_X(r) \lesssim r/\log(r)$.
\end{proof}

\begin{corollary*}
    Let $X$ be a bounded degree graph with exponential growth that satisfies $\sep_X(r)\lesssim r^a$ for some $a<1$. Then
    \[
    \bw_X(r)\simeq r/\log(r)
    \]  
\end{corollary*}
\begin{proof}
    As $X$ has exponential growth, $\kappa_X(r)\simeq \log(r)$, so $\bw_X(r)\gtrsim r/\log(r)$ by Proposition \ref{prop:bwgrowth}. As $\sep_X(r)\lesssim r^a$ for some $a<1$, $\bw_X(r)\lesssim r/\log(r)$ by Proposition \ref{prop:bwsep}.
\end{proof}

\subsection{Graph decompositions}

We start with the quick proofs of Lemma \ref{lem:wid1useless} and Proposition \ref{widinftygrids}.

\begin{proof}[Proof of Lemma \ref{lem:wid1useless}]
The goal is to show that for any graphs $G$ and $X$,
    \[
        \wid^{G,1}_X(r) =\min\{r,|VX|\}
    \]
Let $\Gamma\leq X$ with $|V\Gamma|= r$. For any $G$-decomposition $\{X_g\}_{g\in VG}$ of $\Gamma$, $\bigcup_{g\in G}X_g = V\Gamma$, we have
\[
 \sum_{g\in VG} |X_g| \geq |V\Gamma|=r
\]
Hence, $\wid^{G,1}_X(r) \geq \min\{r,|VX|\}$. The upper bound is obtained from the $G$-decomposition $X_g=V\Gamma$ for one $g\in VG$ and $X_{g'}=\emptyset$ for all $g'\in VG\setminus\{g\}$.    
\end{proof}

Let us recall Proposition \ref{widinftygrids}.

\begin{proposition*}
    Let $G$ be a graph which contains, for each $m\in\N$, a topological subdivision of the $m\times m$ square grid $G_m$\footnote{$VG_m=\{1,\ldots,m\}^2$, $EG_m=\setcon{(a,b)(c,d)}{|a-b|+|c-d|=1}$}. Then for every graph $X$,
    \[
        \wid^{G,\infty}_X(r) \simeq 1.
    \]
\end{proposition*} 
We will prove this using a pair of lemmas. It is likely that both are known to experts.

\begin{lemma} Let $G$ be a graph and let $G'$ be a topological subdivision of $G$. For every graph $X$
\[
    \wid^{G',\infty}_X(r) \leq \wid^{G,\infty}_X(r)
\]    
\end{lemma}
\begin{proof}
    Let $\Gamma$ be a finite graph and let $\{X_g\}_{g\in VG}$ be a $G$-decomposition of $\Gamma$. We define a $G'$-decomposition ${X'_{g'}}_{g'\in VG'}$ of $\Gamma$ as follows:
    \begin{itemize}
        \item if $g'\in VG'$ corresponds to a vertex $g\in VG$, set $X'_{g'}=X_g$
        \item if $g'\in VG'$ is a vertex created while subdiving the edge $g_0g_1\in EG$, set $X'_{g'}=X_{g_0}\cap X_{g_1}$. 
    \end{itemize}
    From the first condition it is clear that $\bigcup_{g'} X'_{g'}=V\Gamma$ and that for every edge $xy\in E\Gamma$ there is some $g'$ such that $x,y\in X'_{g'}$. Now the full subgraph of $G'$ with vertex set $G'_x=\setcon{g'\in VG'}{x\in X'_{g'}}$ is the image under the subdivision of the full subgraph of $G$ with vertex set $G_x=\setcon{g\in VG}{x\in X_{g}}$. The latter full subgraph is connected by assumption, therefore so is the full subgraph with vertex set $G'_x$.
\end{proof}

\begin{lemma}
    Let $\Gamma$ be an $m$-vertex graph. There is a $G_m$ decomposition $\{X_{(a,b)}\}_{1\leq a,b\leq m}$ of $\Gamma$ satisfying $\max\{|X_{(a,b)}|\}\leq \Delta_\Gamma+\max\{\Delta_\Gamma,1\}$
\end{lemma}
 \begin{proof}
     Enumerate $V\Gamma=\{x_1,\ldots,x_m\}$. Define 
     \begin{align*}
        G_{x_k} = & \{(k,k)\}\cup \\ & \bigcup_{\setcon{l<k}{x_kx_l\in E\Gamma}} \{(l,l),(l,l+1),\ldots,(l,k),(l+1,k),\ldots,(k,k)\}\cup \\ & \bigcup_{\setcon{l>k}{x_kx_l\in E\Gamma}} \{(k,k),(k,k+1),\ldots,(k,l),(k,l+1),\ldots,(l,l)\}
     \end{align*}
     We prove that $\{X_{(a,b)}\}_{(a,b)\in VG_m}$ given by $X_{(a,b)}=\setcon{x_k\in V\Gamma}{(a,b)\in G_{x_k}}$ is a $G_m$-decomposition of $\Gamma$.
     It is clear that the full subgraph with vertex set $G_{x_k}$ is connected, $x_k\in X_{(k,k)}$ for each $k$ and $x_k,x_l\in X_{(k,k)}$ whenever $x_kx_l\in E\Gamma$, so $\{X_{(a,b)}\}_{(a,b)\in VG_m}$ is a $G_m$ decomposition of $\Gamma$.

     If $x_k\in X_{(a,b)}$ then either $a=b$ and $X_{(a,a)}=\{x_a\}\cup\setcon{x_k}{x_ax_k\in E\Gamma}$ which has cardinality $1+\deg(x_a)$; or $a\neq b$ and $X_{(a,b)}=\setcon{x_k}{x_ax_k\in E\Gamma}\cup\setcon{x_k}{x_bx_k\in E\Gamma}$ which has cardinality at most $\deg(a)+\deg(b)$. Thus, 
     \[\max_{(a,b)}|X_{(a,b)}| \leq \Delta_\Gamma+\max\{\Delta_\Gamma,1\} \qedhere
     \]
 \end{proof}

\subsection{Coarse wiring}\label{sec:wiring}

In the final part of this section we prove Propositions \ref{prop:widinftywiring} and \ref{prop:wid1wiring}. Let us start with the formal definition of a coarse wiring from \cite{BarrettHume}. The definition given here is slightly different to allow situations where the image graph does not have bounded degree.

\begin{definition}\label{defn:kwiring} Let $\Gamma,\Gamma'$ be graphs. A \textbf{wiring} of $\Gamma$ into $\Gamma'$ is a continuous map $f:\Gamma\to\Gamma'$ which maps vertices to vertices and each edge $xy$ to a walk $W_{xy}$ which starts at $f(x)$ and ends at $f(y)$.

A wiring $f$ is a \textbf{coarse $k$-wiring} if the preimage of each vertex of $V\Gamma'$ contains at most $k$ vertices in $V\Gamma$, and each vertex $v\in V\Gamma'$ is contained in at most $k$ of the walks in $\mathcal W=\setcon{W_{xy}}{xy\in E\Gamma}$.

We consider the \textbf{image} of a wiring $\im(f)$ to be the graph
\[  \phi(V\Gamma)\cup\bigcup_{xy\in E\Gamma} W_{xy} \leq \Gamma'. 
\]
The \textbf{volume} of a wiring $\vol(f)$ is the number of vertices in its image. 

Let $\Gamma$ be a finite graph and let $Y$ be a graph. We denote by $\wir^k(\Gamma\to Y)$ the minimal volume of a coarse $k$-wiring of $\Gamma$ into $Y$. If no such coarse $k$-wiring exists, we say $\wir^k(\Gamma\to Y)=+\infty$.
\end{definition}

\begin{remark}\label{rem:wiringdefns} Any $k$-coarse wiring in the sense of \cite{BarrettHume} is a $k\Delta_{\Gamma'}$-coarse wiring as defined here. When $\Gamma'$ has bounded degree, any $k$-coarse wiring as defined here is a $2k$-coarse wiring in the sense of \cite{BarrettHume}. We write the definition in this way here as we will want to allow $\Gamma'$ to have unbounded degree. 
\end{remark}

We associate two cost functions to the difficulty of coarse wiring one graph into another.

\begin{definition} Let $X$ and $Y$ be graphs. 
\medskip 
We define $\para_{X\to Y}(r)$ to be the minimal $k$ such that every $\Gamma\leq X$ with $|V\Gamma|\leq r$ admits a $k$-coarse wiring into $Y$. For each $k$, we define $\wir^k_{X\to Y}(r)=\max\setcon{\wir^k(\Gamma\to Y)}{\Gamma\leq X,\ |V\Gamma|\leq r}$.
\end{definition}
Note that when $X$ has bounded degree, the upper bound $\para_{X\to Y}(r)\leq \Delta_X r$ is obtained by mapping all of $\Gamma$ to a single vertex.

\begin{remark}
    A coarse wiring can be seen as a generalisation of a regular map as, when restricted to the vertex set, it need not be Lipschitz.
    A regular map between bounded degree graphs can always be extended to a coarse wiring (using exactly the construction from \S\ref{sec:buildGamma'}), meaning that whenever there is a $\kappa$-regular map $X\to Y$, and $X$ has infinitely many vertices, $\wir^k_{X\to Y}(r) \simeq r$ for all $k$ larger than some constant $k_0=k_0(\kappa,\Delta_X,\Delta_Y)$. The converse is not true, $\wir^1_{T_3\to \Z^2}(r) \simeq r$ but there is no regular map $T_3\to\Z^2$ as the former has exponential growth and the latter has quadratic growth. Typically, the restriction of a coarse wiring to the vertex set is not Lipschitz (even on average).
\end{remark}

We now prove Proposition \ref{prop:widinftywiring}

\begin{proposition*}
    Let $X$ and $G$ be graphs. Then for every $r\in\N$
    \[
        \wid^{G,\infty}_X(r) \leq 2 \para_{X\to G}(r)
    \]
    If, in addition, $X$ has bounded degree, then for every $r\in\N$
    \[
        \para_{X\to G}(r) \leq \max\{1,\Delta_X\}\wid^{G,\infty}_X(r)
    \]
\end{proposition*}
\begin{proof} Let $\Gamma\leq X$ with $|V\Gamma|\leq r$ and let $f:\Gamma\to G$ be a $k$-coarse wiring with $k\leq \para_{X\to G}(r)$.

For each $x\in V\Gamma$, set $\overline{f}(x)$ to be the set of all vertices in $\bigcup_{xy\in E\Gamma} W_{x,y}$. We claim that $X_g=\setcon{x\in V\Gamma}{g\in \overline{f}(x)}$ defines a $G$-decomposition of $\Gamma$.

For each edge $xy\in E\Gamma$, $\{x,y\}\subset X_g$ for any $g\in VW_{xy}\neq\emptyset$. Next, by construction, for each $x\in V\Gamma$,
\[
 \setcon{g\in G}{x\in X_g}=\overline{f}(x)
\]
which induces a connected subgraph ($\bigcup_{xy\in E\Gamma} W_{x,y}$) of $G$.

Finally, for each $g$, there are at most $k$ edges $xy\in E\Gamma$ such that $g\in W_{xy}$ and therefore at most $2k$ vertices $x\in V\Gamma$, such that $g\in\overline{f}(x)$. Thus
\[
    \wid^{G,\infty}_X(r)\leq 2\para_{X\to G}(r)
\]
Conversely, let $\{X_g\}_{g\in VG}$ be a $G$-decomposition of $\Gamma$ with $\max_{g\in VG} |X_g| \leq k$. If $\Delta_X=0$ then 
\[\wid^{G,\infty}_X(r) = \para_{X\to G}(r) = 
\left\{
\begin{array}{cc}
    \left\lceil\frac{r}{|G|}\right\rceil & \textup{if }|G|<+\infty, \\
    1 & \textup{if }|G|=+\infty
\end{array}
\right.
\]
Now assume $\Delta_X\geq 1$. Let us define a wiring $f:\Gamma\to G$ as follows:
\begin{itemize}
    \item for each $x\in V\Gamma$ choose $f(x)\in G_x$,
    \item for each $xy\in E\Gamma$, $G_x, G_y$ are connected and intersect, so set $W_{xy}$ to be any path composed of a path from $f(x)$ to some $z\in G_x\cap G_y$ contained in $G_x$ followed by a path from $z$ to $f(y)$ contained in $G_y$.
\end{itemize}
Let $g\in VG$. If $g\in W_{xy}$ for some $xy\in E\Gamma$, then $g\in G_x\cup G_y$, so one of $x,y$ is contained in $X_g$. Therefore
\[
 \max_{g\in VG} |\setcon{xy\in E\Gamma}{g\in W_{xy}}| \leq \Delta_\Gamma\max_{g\in VG}|X_g|
\]
As this holds for every $\Gamma$ and every $G$-decomposition, we have
\[
    \wid^{G,\infty}_X(r) \leq \max\{1,\Delta_X\}\para_{X\to G}(r)
\]
\end{proof}

Finally, we prove Proposition \ref{prop:wid1wiring}. 

\begin{proposition*}
    For every $k$ and every finite graph $\Gamma$
    \begin{itemize}
        \item if there is a coarse $k$-wiring $f:\Gamma\to G$ with volume at most $V$, then ${}^{2k}\wid^{G,1}(\Gamma)\leq 2kV$,
        \item if ${}^k\wid^{G,1}(\Gamma)\leq V$ then there is a coarse $\Delta_\Gamma k$-wiring $\Gamma\to G$ with volume at most $(1+\Delta_\Gamma)V$.
    \end{itemize}
\end{proposition*}
\begin{proof} 
We recall that ${}^k\wid^{G,1}(\Gamma)\leq V$ if and only if there is a $G$-decomposition $\{X_g\}_{g\in VG}$ of $\Gamma$ with the properties
\begin{itemize}
    \item $\max_{g\in VG} |X_g| \leq k$,
    \item $\sum_{g\in VG} |X_g| \leq V$,
\end{itemize}

Choose a coarse $k$-wiring $f:\Gamma\to G$ with volume at most $V$. We define a $G$-decomposition $\{X_g\}_{g\in VG}$ exactly as in the proof of Proposition \ref{prop:widinftywiring}, and deduce using that proposition that
\[
 \max_{g\in VG} |X_g| \leq 2k
\]
Consequently,
\begin{align*}
    \sum_{g\in VG} |X_g| & \leq |\setcon{g\in VG}{X_g\neq\emptyset}|\max_{g\in VG} |X_g| \\
        & \leq \vol(f) \max_{g\in VG} |X_g| \\
        & \leq 2 k V 
\end{align*}

    We now build a $\Delta_Xk$-coarse wiring $f:\Gamma\to G$ exactly as the proof of Proposition \ref{prop:widinftywiring}. Note that for each $x,y$, $|W_{xy}|\leq |G_x|+|G_y|$. Also,
    \[
    \sum_{g\in VG} |X_g| = \sum_{x\in V\Gamma} |G_x|
    \]
    Combining these observations, we see that
    \begin{align*}
        \vol(f) & \leq |V\Gamma| + \sum_{xy\in E\Gamma} |W_{xy}| \\
                & \leq \sum_{g\in VG} |X_g| + \sum_{xy\in E\Gamma} (|G_x|+|G_y|) \\
                & \leq (1+\Delta_X) \sum_{g\in VG} |X_g| \\
                & \leq (1+\Delta_X)V
    \end{align*} \qedhere
\end{proof}

\section{Questions}
We start by repeating a question raised earlier in the paper.

\begin{question}\label{qu:septw}
    Is there a bounded degree graph $X$ such that $\sep_X\not\simeq\tw_X$?
\end{question}

Also, we believe that the following is open.

\begin{question}
    Let $T_3,T$ be the regular trees of degree $3$ and $+\infty$ respectively. Does
    \[
     \wid^{T_3,p}_X(r)\not\simeq \wid^{T,p}_X(r)
    \]
    hold for every bounded degree graph $X$?
\end{question}

Other natural questions in this direction include finding pairs of graphs $X,Y$ and $1\leq p, q \leq +\infty$ such that
\begin{itemize}
    \item $\cw^p_X\simeq \cw^p_Y$ but $\cw^q_X\not\simeq \cw^q_Y$,
    \item $\bw^p_X\simeq \bw^p_Y$ but $\bw^q_X\not\simeq \bw^q_Y$,
    \item $\wid^{G,p}_X\simeq \wid^{G,p}_Y$ but $\wid^{G,q}_X\not\simeq \wid^{G,q}_Y$ (with $p>1$).
\end{itemize}
Given the near-optimal bounds on $\cw^p$ provided by Theorem \ref{sepboundscw}, the first of these would be particularly interesting (if such an example exists).
\medskip

Perhaps the easiest case where the equivalence $\sep_X\simeq\cw_X$ is still open is for regular tesselations of the hyperbolic plane. A more general question is the following.

We recall that a graph $X$ is called \textbf{non-amenable} if there is some $\varepsilon>0$ such that, for every finite subset $A\subset VX$
\[
    |\setcon{v\in VX}{d_X(A,v)=1}|\geq \varepsilon|A|
\]
\begin{question}\label{qu:hypplane}
    Does every bounded degree non-amenable planar graph $X$ satisfy $\cw_X(r)\lesssim \log(r)$?
\end{question}
This upper bound holds for metric balls in non-amenable planar graphs \cite{KMP-planar-Glauber}.\footnote{In this paper the graphs are called hyperbolic rather than non-amenable, but hyperbolic is used elsewhere in this paper to mean Gromov hyperbolic.}

As mentioned previously, we know much less about bandwidth than cutwidth, so the following question remains very interesting.
\begin{question}
    For which (Cayley) graphs does 
    \[\bw_X(r)\simeq r/\kappa_X(r)\]
    hold?
\end{question}
Regular trees of finite degree $\geq 3$ and Cayley graphs of non-elementary hyperbolic groups are examples. A natural collection of graphs to consider would be vertex transitive graphs of polynomial growth (as the corresponding result holds for both separation and cutwidth profiles). Another natural result which may hold is

\begin{question}
    Does $\bw_X(r)\lesssim r/\kappa_X(r)$
    hold for every bounded degree graph with finite Assouad-Nagata dimension?
\end{question}
The corresponding result for the separation profile (and the definition of Assouad-Nagata dimension) appears in \cite[Theorem 1.5]{HumSepExp}.
\medskip

For other values of $p$ we offer two (motivated) guesses for trees and integer grids:

\begin{question}\label{qu:bwptree}
    Is it true that for every $p\in[1,\infty)$
    \[
     \bw^p_{T_3}(r)\simeq \left\{ \begin{array}{ll}
         r\log(1+r) & p=1 \\
         r & p>1 
     \end{array}\right.
    \]
    where $T_3$ is the infinite $3$-regular tree?
\end{question}
The rationale behind this suggestion is that rooted binary trees should be (among the) hardest subgraphs to order effectively, and that for the binary tree of depth $r$ (with $2^{r+1}-1$ vertices) there is an ordering where, for each $1\leq i \leq r$, there are exactly $2^i$ edges whose end vertices are $2^{r-1}$ apart in the ordering. Applying the same rationale to $p$-cutwidth yields the guess

\begin{question}
    Let $T_3$ be the infinite $3$-regular tree. Does
    \[
     \cw^p_{T_3}(r)\simeq r^{\frac1p}\log(1+r)
    \]
    hold for every $p\in[1,\infty)$?
\end{question}
In this case the upper bound already holds thanks to Proposition \ref{prop:compare}.

\begin{question}\label{qu:bwpgrid}
    Let $\Z^k$ be the $k$-dimensional integer grid. Does
    \[
     \bw^p_{\Z^k}(r)\simeq r^{\frac1p+(1-\frac1k)}
    \]
    hold for every $p\in[1,\infty)$?
\end{question}
Again, the rationale is that cubes are among the hardest subgraphs to order, and that the best ordering is lexicographic, meaning that the end vertices of every edge are at most $r^{1-\frac1k}$ apart (and a positive proportion of them are exactly that far apart).
\medskip

Another very natural question is
\begin{question} Does $\cw^p_X\lesssim \bw^p_X$ hold
for every bounded degree graph $X$, and every $p\in(1,\infty)$,
\end{question}

Finally, we make one guess to try to generate interest in a broader collection of the graph decomposition invariants. Typically, the subgraphs of an integer grid $\Z^d$ which have the highest cost resemble $d$-dimensional cubes. We can decompose these cubes over $\Z^k$ (with $k\leq d$) in the following way:
\[
 X_{(x_1,\ldots,x_k)}= \setcon{(y_1,\ldots,y_d)}{\sum_{i=1}^k |x_i-y_i|\leq 1}
\]
Starting with a $d$-cube $C^d_r$ with $r^d$ vertices, we have $|X_g|\simeq r^{d-k}$ for $\simeq r^k$ vertices, and $|X_g|=0$ elsewhere. Hence, for every $p\in [1,\infty)$
\[
    \wid^{\Z^k,p}(C^d_r) \lesssim_{k,d} r^k r^{\frac{d-k}{p}}
\]
This motivates the following question
\begin{question}
    Does
    \[
    \wid^{\Z^k,p}_{\Z^d}(r) \simeq_{k,d} r^{\frac{d+k(p-1)}{dp}}
    \]
    hold for every $p\in[1,\infty)$ and every $1\leq k < d$?
\end{question}

\def\cprime{$'$}


\begin{thebibliography}{BPTW10}

\bibitem[ACP87]{ACP-twNP}
Stefan Arnborg, Derek~G. Corneil, and Andrzej Proskurowski.
\newblock Complexity of finding embeddings in a k-tree.
\newblock {\em SIAM Journal on Algebraic Discrete Methods}, 8(2):277--284,
  1987.

\bibitem[AST90]{Alon-Seymour-Thomas}
Noga Alon, Paul Seymour, and Robin Thomas.
\newblock A separator theorem for nonplanar graphs.
\newblock {\em Journal of the American Mathematical Society}, 3(4):801--808,
  1990.

\bibitem[BH21]{BarrettHume}
Benjamin Barrett, David Hume, Larry Guth and Elia Portnoy.
\newblock Thick embeddings of graphs into symmetric spaces via coarse geometry.
\newblock {\em Trans.\ Amer.\ Math.\ Soc.} 378 (2025), 885--909.

\bibitem[BKMP01]{KMP-planar-Glauber}
Noam Berger, Claire Kenyon, Elchanan Mossel, and Yuval Peres.
\newblock Glauber dynamics on trees and hyperbolic graphs.
\newblock {\em Probab. Theory Relat. Fields} 131, 311–340 (2005).

\bibitem[Bod88]{Bodlaender_Classes}
Hans Bodlaender.
\newblock Some classes of graphs with bounded treewidth.
\newblock {\em Bulletin of the European Association for Theoretical Computer
  Science (EATCS)}, 36, 01 1988.

\bibitem[Bod06]{Bodlaender_treewidth}
Hans~L. Bodlaender.
\newblock Treewidth: Characterizations, applications, and computations.
\newblock In Fedor~V. Fomin, editor, {\em Graph-Theoretic Concepts in Computer
  Science}, pages 1--14, Berlin, Heidelberg, 2006. Springer Berlin Heidelberg.

\bibitem[BPTW10]{BPTW}
Julia Böttcher, Klaas P. Pruessmann, Anusch Taraz and Andreas Würfl.
\newblock Bandwidth, expansion, treewidth, separators and universality for bounded-degree graphs.
\newblock {\em European Journal of Combinatorics}, 31(5):1217--1227, 2010.

\bibitem[BST12]{BenSchTim-12-separation-graphs}
Itai Benjamini, Oded Schramm, and {\'A}dam Tim{\'a}r.
\newblock On the separation profile of infinite graphs.
\newblock {\em Groups Geom. Dyn.}, 6(4):639--658, 2012.

\bibitem[Coz20]{LeCozDiagonal}
Corentin~Le Coz.
\newblock Poincar\'e profiles of lamplighter diagonal products.
\newblock {\em Preprint, \tt{arXiv:2007.04709}}, 2020.

\bibitem[CS89]{Chung-Seymour}
Fan~R.~K. Chung and Paul~D. Seymour.
\newblock Graphs with small bandwidth and cutwidth.
\newblock {\em Discrete Mathematics}, 75(1):113--119, 1989.

\bibitem[DPS02]{DPS-survey}
Josep D\'{\i}az, Jordi Petit, and Maria Serna.
\newblock A survey of graph layout problems.
\newblock {\em ACM Comput. Surv.}, 34(3):313–356, sep 2002.

\bibitem[DV03]{Djidjev2003CrossingNA}
Hristo~N. Djidjev and Imrich Vrto.
\newblock Crossing numbers and cutwidths.
\newblock {\em J. Graph Algorithms Appl.}, 7:245--251, 2003.

\bibitem[Fr21]{Frances-coarse}
Charles Frances.
\newblock Isometry group of Lorentz manifolds: A coarse perspective
\newblock {\em GAFA}, 31(5):1095--1159, 2021.

\bibitem[GC23]{LeCozGournay}
Antoine Gournay and Corentin Le~Coz.
\newblock Separation profile, isoperimetry, growth and compression.
\newblock {\em Annales de l'Institut Fourier}, 73(4):1627--1675, 2023.

\bibitem[GJ79]{GJ79}
Michael~R. Garey and David~S. Johnson.
\newblock Computers and intractability. {A} guide to the theory of
  {NP}-completeness.
\newblock A {Series} of {Books} in the mathematical {Sciences}. {San}
  {Francisco}: {W}. {H}. {Freeman} and {Company}. {X}, 338 p. (1979)., 1979.

\bibitem[GP23]{GeorPapo}
  Agelos Georgakopoulos and Panagiotis Paposoglu.
  \newblock Graph minors and metric spaces.
  \newblock To appear in Combinatorica, available from arXiv:2305.07456.

\bibitem[GS21]{GladShum}
Valeriia Gladkova and Verna Shum.
\newblock Separation profiles of graphs of fractals.
\newblock {\em Journal of Topology and Analysis}, 13(04):1111--1123, 2021.

\bibitem[Gui73]{Guivarch}
Yves~Guivarc'h.
\newblock Croissance polynomiale et p\'eriodes des fonctions harmoniques.
\newblock {\em Bull. Soc. Math. France}, 101:333--379, 1973.

\bibitem[HM20]{HumeMack}
David Hume and John~M. Mackay.
\newblock Poorly connected groups.
\newblock {\em Proc. Amer. Math. Soc.}, 148(11):4653--4664, 2020.

\bibitem[HMT20]{HumeMackTess-Pprof}
David Hume, John~M. Mackay, and Romain Tessera.
\newblock Poincar\'{e} profiles of groups and spaces.
\newblock {\em Rev. Mat. Iberoam.}, 36(6):1835--1886, 2020.

\bibitem[HMT22]{HumeMackTess-PprofLie}
David Hume, John~M. Mackay, and Romain Tessera.
\newblock Poincar\'e profiles of {L}ie groups and a coarse geometric dichotomy.
\newblock {\em Geom. Funct. Anal.}, 32:1063--1133, 2022.

\bibitem[HMT25]{HumeMackTess-genasdim}
David Hume, John~M. Mackay, and Romain Tessera.
\newblock Asymptotic dimension for covers with controlled growth.
\newblock {\em J.\ London Math.\ Soc.}, 111:e70043. https://doi.org/10.1112/jlms.70043, 2025.

\bibitem[Hum17]{HumSepExp}
David Hume.
\newblock A continuum of expanders.
\newblock {\em Fund. Math.}, 238:143--152, 2017.

\bibitem[MPS83]{MPS83}
Fillia Makedon, Christos~H Papadimitriou, and Ivan Sudborough.
\newblock Topological bandwidth.
\newblock {\em Proceedings on the Colloquium on Trees in Algebra and
  Programming}, pages 317--331, 01 1983.

\bibitem[MS89]{MS89}
Fillia Makedon and Ivan~Hal Sudborough.
\newblock On minimizing width in linear layouts.
\newblock {\em Discrete Applied Mathematics}, 23(3):243--265, 1989.

\bibitem[Rai23]{Raistrick-wiring}
R.~Raistrick.
\newblock Dependence of the coarse wiring profile on the choice of parameter.
\newblock {\em Preprint, {\tt arxiv:2311.08436}}, 2023.

\bibitem[Tr85]{Trofimov}
V. I. Trofimov. 
\newblock Graphs with polynomial growth.
\newblock {\em Math. USSR-Sb.}, 51:405--417, 1985.

\end{thebibliography}
\end{document}